\documentclass[10pt]{amsart}

	\usepackage[pagebackref,  pdfpagemode=FullScreen,  colorlinks=true]{hyperref}	   

	\usepackage{verbatim, graphicx, subfig, psfrag, mathrsfs}

	\newtheorem{thm}{Theorem}[section]
  	\newtheorem{cor}{Corollary}[section]
  	\newtheorem{lem}{Lemma}[section]
  	\newtheorem{prop}{Proposition}[section]

	\theoremstyle{definition}

	\theoremstyle{remark}
  	\newtheorem{rem}[thm]{Remark}
	
	\makeatletter
	\@namedef{subjclassname@2010}{%
	\textup{2010} Mathematics Subject Classification}
	\makeatother

\newcommand{\MM}{\mathcal M}

\newcommand{\BM}{\overline{\mathcal M}}

\newcommand{\FF}{\mathcal F}

\newcommand{\OO}{\mathcal O}

\newcommand{\HH}{\mathcal H}

\newcommand{\hyp}{\operatorname{hyp}}

\newcommand{\even}{\operatorname{even}}
\newcommand{\odd}{\operatorname{odd}}

\newcommand{\GL}{\operatorname{GL}}

\newcommand{\Pic}{\operatorname{Pic}}

\newcommand{\bbZ}{\mathbb Z}
\newcommand{\bbP}{\mathbb P}
\newcommand{\bbR}{\mathbb R}

\begin{document}

\title{Loci of curves with subcanonical points in low genus}
\author{Dawei Chen}
\address{Department of Mathematics, Boston College, Chestnut Hill, MA 02467}
\email{dawei.chen@bc.edu}

\author{Nicola Tarasca}
\address{Department of Mathematics, University of Utah, Salt Lake City, UT 84112}
\email{tarasca@math.utah.edu}

\subjclass[2010]{ 14H99 (primary),  14C99 (secondary)}
\keywords{\makebox[4.46in][s]{Higher Weierstrass points, spin curves, minimal strata of Abelian differentials,} \newline \indent  effective cycles in moduli spaces of curves}

%\date{\today}

\thanks{\makebox[5.74in][s]{During the preparation of this article the first author was partially supported by NSF grant DMS-1200329} \newline \indent and NSF CAREER grant DMS-1350396.}

\begin{abstract}
Inside the moduli space of curves of genus three with one marked point, we consider the locus of hyperelliptic curves with a marked Weierstrass point, and the locus of non-hyperelliptic curves with a marked hyperflex. These loci have codimension two. We compute the classes of their closures in the moduli space of stable curves of genus three with one marked point.
Similarly, we compute the class of the closure of the locus of curves of genus four with an even theta characteristic vanishing with order three at a certain point. These loci naturally arise in the study of minimal dimensional strata of Abelian differentials. 
\end{abstract}

\maketitle

\section{Introduction}
\label{intro}

A point $p$ on a smooth curve of genus $g\geq 2$ is called {\it subcanonical} if $(2g-2)p$ is a canonical divisor.
For example, Weierstrass points of a hyperelliptic curve are subcanonical points.
Subcanonical points are the most special Weierstrass points from the perspective of Weierstrass gap sequences (\cite[Exercise 1.E]{MR770932}).  
The locus of curves admitting a subcanonical point inside the moduli space $\MM_g$ of curves of genus $g$ has codimension $g-2$. 

Such loci naturally arise from the study of strata of Abelian differentials. Let $\mathscr{H}_g$ be the moduli space of Abelian differentials parameterizing pairs $(C, \omega)$, where $C$ is a smooth curve of genus $g$ and $\omega$ is an Abelian differential on $C$, for $g\geq 2$. Given $\mu = (m_1, \ldots, m_n)$ a partition of $2g-2$, denote by $\mathscr{H}(\mu) \subseteq \mathscr{H}_g$ the stratum of Abelian differentials $(C, \omega)$, where the zeros of $\omega$ are of type $\mu$, namely, $(\omega)_0 = \sum_{i=1}^n m_i p_i$, for distinct points $p_i$. There is a $\GL^+_2(\bbR)$-action on $\mathscr{H}(\mu)$ that varies the real and imaginary parts of $\omega$ (\cite{MR2261104}). The study of $\GL^+_2(\bbR)$-orbits is a major subject in Teichm\"uller dynamics, with fascinating applications to the geometry of the moduli space of stable curves, such as producing rigid curves, and bounding slopes of effective divisors (\cite{MR2822219}). Moreover, one can degenerate $(C, \omega)$ to a nodal curve along with a section of the dualizing sheaf, that is, compactify $\mathscr{H}(\mu)$ over the Deligne-Mumford moduli space $\BM_g$ of 
stable nodal curves, and study its boundary behavior. This turns out to provide crucial information for invariants in Teichm\"uller dynamics, such as Siegel-Veech constants, Lyapunov exponents, and classification of Teichm\"uller curves (\cite{MR3033521, MR2910796}). 

Among all strata of Abelian differentials, the minimal dimensional stratum is $\mathscr{H}(2g-2)$, parameterizing pairs $(C,\omega)$, where $(\omega)_0 =(2g-2)p$ for a certain point $p$ in $C$. Note that if $(\omega)_0 =(2g-2)p$, then $\mathcal{O}_C((g-1)p)$ is a {\it spin structure} on $C$, that is, a square root of the canonical bundle. The parity of a spin structure $\eta$ on a curve $C$ is defined as the parity of $h^0(C,\eta)$. 

For $g\geq 4$, the stratum $\mathscr{H}(2g-2)$ has three connected components: the component of hyperelliptic curves $\mathscr{H}(2g-2)^{\hyp}$, and the two components $\mathscr{H}(2g-2)^{\even}$ and $\mathscr{H}(2g-2)^{\odd}$ distinguished by the parity of the corresponding spin structures.
For $g=3$, the hyperelliptic and even components coincide, hence $\mathscr{H}(4)$ has two connected components: $\mathscr{H}(4)^{\hyp}$ and $\mathscr{H}(4)^{\odd}$. For $g=2$, the stratum $\mathscr{H}(2)$ is irreducible. The reader can refer to \cite{MR2000471} for a complete classification of connected components of $\mathscr{H}(\mu)$, where the minimal dimensional stratum $\mathscr{H}(2g-2)$ plays an important role as a base case for an inductive argument. 

One can project $\mathscr{H}(\mu)$ to $\MM_g$ via the forgetful map $(C,\omega)\mapsto C$. Alternatively, by marking the zeros of $\omega$, one can lift 
$\mathscr{H}(\mu)$ (up to rescaling $\omega$) to $\MM_{g,n}$. Thus, one obtains a number of interesting subvarieties in $\MM_g$ and $\MM_{g,n}$. For example, 
the projection of $\mathscr{H}(2)$ dominates $\MM_2$, and the lift of $\mathscr{H}(2)$ is the divisor of Weierstrass points in $\MM_{2,1}$.

Motivated by the problem of determining the Kodaira dimension of moduli spaces of stable curves $\BM_{g,n}$, effective divisor classes in $\BM_{g,n}$ have been extensively calculated in the last several decades. In contrast, much less is known about higher codimensional cycles in $\BM_{g,n}$. Recently, there has been growing interest in studying higher codimensional subvarieties of moduli spaces of curves (\cite{MR2120989, Hain, FaberPagani, MR3109733, DR}),
and in describing cones of higher codimensional effective cycles (\cite{CC}). 
Our main result is the explicit computation of classes of closures of loci of curves with subcanonical points in the moduli space of stable curves in low genus.

Let $\mathcal{H}yp_{3,1}$ be the locus of smooth hyperelliptic curves with a marked Weierstrass point in $\MM_{3,1}$. Let 
$\FF_{3,1}$ be the locus of smooth non-hyperelliptic curves with a marked hyperflex point in $\MM_{3,1}$. The loci $\mathcal{H}yp_{3,1}$ and $\FF_{3,1}$ have codimension two in $\MM_{3,1}$. They are the lifts of the minimal dimensional stratum components $\mathscr{H}(4)^{\hyp}$ and $\mathscr{H}(4)^{\odd}$, respectively. 
In  \S \ref{pullbacktoMbar31} and \S \ref{SF} we prove the following result in the Chow group $A^2(\BM_{3,1})$ of codimension-two classes on $\BM_{3,1}$.

\begin{thm}
\label{res31}
The classes of the closures of $\mathcal{H}yp_{3,1}$ and $\FF_{3,1}$ in $A^2(\BM_{3,1})$ are
\begin{eqnarray*}
\overline{\mathcal{H}yp}_{3,1} & \equiv &  \psi (18\lambda -2\delta_0 -9\delta_{1,1} -6\delta_{2,1})-\lambda\left(45\lambda -\frac{19}{2}\delta_0 -24\delta_{2,1}\right) -\frac{1}{2}\delta_0^2 -\frac{5}{2}\delta_0\delta_{2,1} -3\delta_{2,1}^2,\\
\overline{\mathcal{F}}_{3,1} & \equiv &{} \psi(77\lambda -3\psi -8\delta_0 -42\delta_{1,1} -19\delta_{2,1}) - \lambda\left(338\lambda -\frac{137}{2}\delta_0  -146 \delta_{2,1}\right) -\frac{7}{2}\delta_0^2\\
&& {}  -\frac{31}{2}\delta_0\delta_{2,1} -3\delta_{1,1}^2  -20\delta_{2,1}^2 +3\kappa_2.
\end{eqnarray*}
\end{thm}

As a check, consider the map $p\colon \BM_{3,1} \rightarrow \BM_3$ obtained by forgetting the marked point. The push-forward of products of divisor classes in $\BM_{3,1}$ via $p$ are described in \S \ref{pushfwdtoM3}.
It follows that the push-forward of the class of $\overline{\mathcal{H}yp}_{3,1}$ from Theorem \ref{res31} via $p$ is 
\[
p_*\left(\overline{\mathcal{H}yp}_{3,1} \right) \equiv 8\cdot \overline{\mathcal{H}yp}_{3},
\]
as expected, where $\overline{\mathcal{H}yp}_{3}$ is the closure of the locus of hyperelliptic curves in $\BM_3$. Indeed, from \cite{MR664324} we have $\overline{\mathcal{H}yp}_{3}\equiv 9\lambda-\delta_0-3\delta_1$, and the multiplicity $8$ accounts for the number of Weierstrass points in a hyperelliptic curve of genus $3$.  The class of the divisor $p_*(\overline{\mathcal{F}}_{3,1})$ is computed in \cite{MR1016424}, and will be used in the proof of Theorem \ref{res31}.

Similarly, we can consider the following codimension-two loci in the moduli space of curves of genus $4$.
Let $\HH_4$ be the locus in $\MM_4$ of smooth curves $C$ that admit a canonical divisor of type $6p$. It consists of the three irreducible components 
$\overline{\mathcal{H}yp}_{4}$, $\HH_4^{+}$ and $\HH_4^{-}$, defined as the projection images of the minimal dimensional stratum components $\mathscr{H}(6)^{\hyp}$, 
$\mathscr{H}(6)^{\even}$, and $\mathscr{H}(6)^{\odd}$ to $\BM_{4}$, respectively. 

The class of the closure of $\overline{\mathcal{H}yp}_{4}$ has been first computed in \cite[Proposition 5]{MR2120989} via localization on the moduli space of stable relative maps (see (\ref{hyp4})). An alternative proof via test surfaces and admissible covers is given in \cite{MR3109733}. In \S \ref{h+4} we obtain the following result in the Chow group $A^2(\BM_4)$.

\begin{thm}
\label{h4+}
The class of the closure of the locus $\HH_4^{+}$ in $A^2(\BM_4)$ is
\begin{eqnarray*}
\overline{\mathcal{H}}^+_4 & \equiv & 2448\lambda^2 -542 \lambda\delta_0 - 1608 \lambda\delta_1  + 276\lambda\delta_2 + 32 \delta_0^2 + 178\delta_0\delta_1  \\
& & {}+ 336\delta_1^2   + 276 \delta_1\delta_2 + 576 \delta_2^2 - 4\delta_{00} - 60 \gamma_1 + 12 \delta_{01a} - 144 \delta_{1|1}. 
\end{eqnarray*}
\end{thm}

As an immediate consequence of Theorems \ref{res31} and \ref{h4+}, in \S \ref{ci} we prove the following geometric result. 

\begin{cor}
\label{coro}
The loci $\overline{\mathcal{H}yp}_{3,1}$, $\overline{\mathcal{F}}_{3,1}$, and $\overline{\mathcal{H}}^+_4$ are not complete intersections in their respective spaces.
\end{cor}

The class of $\overline{\mathcal{H}yp}_{3,1}$ is known to span an extremal ray of the cone of effective codimension-two classes in $\BM_{3,1}$ (similarly for the locus $\overline{\mathcal{H}yp}_{4}$ in $\BM_4$) (\cite{CC}). It is natural to ask whether the classes of $\overline{\mathcal{F}}_{3,1}$ and $\overline{\mathcal{H}}^+_4$ lie in the interior or in the boundary of the cones of effective codimension-two classes of their respective spaces. In addition, the loci $\overline{\mathcal{H}yp}_{3,1}$, $\overline{\mathcal{F}}_{3,1}$, and $\overline{\mathcal{H}}^+_4$ are images of the strata $\mathscr{H}(4)^{\hyp}$, $\mathscr{H}(4)^{\odd}$, and $\mathscr{H}(6)^{\even}$ in the respective moduli spaces of curves. Thus our results can shed some light on the study of  cones of effective higher codimensional cycles, as well as %Teichm\"uller dynamics for these strata and 
the study of degenerate Abelian differentials appearing in the boundary of these strata.

Let us describe our methods. In order to study the closure of the images of the strata of Abelian differentials inside moduli spaces of curves, one needs a good description of singular hyperelliptic curves, and singular curves with an even or odd spin structure. The closure of the hyperelliptic components can be studied via the theory of admissible covers (\cite{MR664324}). Moreover, let $\mathcal{S}_g^+$ (respectively, $\mathcal{S}_g^-$) be the moduli space of pairs $[C,\eta]$, where $C$ is a smooth curve of genus $g$, and $\eta$ is an even (respectively, odd) {\it theta characteristic} on $C$, that is, a spin structure on $C$. There is a natural map $\mathcal{S}_g^\pm \rightarrow \MM_g$ obtained by forgetting the theta characteristic. We  use Cornalba's description of the closure $\overline{\mathcal{S}}_g^\pm \rightarrow \BM_g$ (\cite{MR1082361}, see also \cite{MR2779475, MR2565536, MR3245010}), and realize the loci $\overline{\mathcal{F}}_{3,1}$ and $\overline{\mathcal{H}}^+_4$ as push-forward of loci in $\overline{\mathcal{S}}_3^-\times_{\BM_3} \BM_{3,1}$ and $\overline{\mathcal{S}}_4^+$, respectively.

To perform our computations, we use the basis for $A^2(\BM_{4})$ from \cite{MR1078265}. 
It is not clear whether the group of tautological classes $R^2(\BM_{3,1})\subseteq A^2(\BM_{3,1})$ coincides with $A^2(\BM_{3,1})$. Anyhow, it is a consequence of the descriptions in Lemmata \ref{j3^*H} and \ref{lemmamnklj} that the classes of 
$\overline{\mathcal{H}yp}_{3,1}$ and $\overline{\mathcal{F}}_{3,1}$
are tautological.
In \S \ref{basisforMbar31}, we describe a basis for $R^2(\BM_{3,1})$. 
The classes $\psi, \lambda, \delta_0, \delta_{1,1}, \delta_{2,1}$ form a $\mathbb{Q}$-basis of $\Pic(\BM_{3,1})$. 
Note that the product $\lambda\delta_{1,1}$ is linearly dependent from the other products, see Proposition \ref{boundaryinA2Mbar31}.
%Let $\delta_{01a}$ in $R^2(\BM_{3,1})$ be the class of the closure of the locus of curves with a nodal rational tail attached to a component of genus $2$ containing the marked point. 

\begin{prop}
\label{basisR2Mbar31}
The following $16$ classes
\begin{align}
\label{bm31gen}
\psi^2, \psi\lambda, \psi\delta_0, \psi\delta_{1,1}, \psi\delta_{2,1}, \lambda^2, \lambda\delta_0, \lambda\delta_{2,1}, \delta_0^2,
\delta_0\delta_{1,1}, \delta_0\delta_{2,1}, \delta_{1,1}^2 , \delta_{1,1}\delta_{2,1}, \delta_{2,1}^2, \delta_{01a}, \kappa_2
\end{align}
form a $\mathbb{Q}$-basis of $R^2(\BM_{3,1})$.
\end{prop}

The paper is organized as follows. We collect in \S \ref{ingre} some results about the enumerative geometry of a general curve. In \S \ref{basisforMbar31}, we prove Proposition \ref{basisR2Mbar31} using the following strategy. After \cite{getzler-looijenga, yang}, it is known that the group $R^2(\BM_{3,1})$ has dimension $16$, equal to the rank of the intersection pairing $R^2(\BM_{3,1})\times R^{5}(\BM_{3,1})$. Hence, one could show that the intersection pairing of classes in (\ref{bm31gen}) with classes of complementary dimension has rank $16$. Instead, we reduce the number of computations by using Getzler's results on $R^2(\BM_{2,2})$ (\cite{MR1672112}). Let $\vartheta\colon \BM_{2,2}\rightarrow \BM_{3,1}$ be the map obtained by attaching a fixed elliptic tail to the second marked point.
We first show that the pull-backs of the classes in (\ref{bm31gen}) via $\vartheta$ span a $13$-dimensional subspace of the $14$-dimensional space $R^2(\BM_{2,2})$. We compute the $3$-dimensional space of possible relations among the classes in (\ref{bm31gen}). Finally, by restricting to three test surfaces, we verify that such relations cannot hold.

In \S \ref{pullbacktoMbar31} we compute the class of $\overline{\mathcal{H}yp}_{3,1}$ by studying the pull-back of the class $\overline{\mathcal{H}yp}_{4}$ via the map $j_3\colon \BM_{3,1}\rightarrow \BM_4$ obtained by attaching a fixed elliptic tail at the marked point of curves in $\BM_{3,1}$. To compute the classes of $\overline{\mathcal{F}}_{3,1}$ and $\overline{\mathcal{H}}^+_4$, in \S \ref{SF} and \S \ref{h+4} we  realize each one of these loci as a component of the intersection of two divisors in their respective moduli spaces. We  first describe set-theoretically the other components in the intersections. The multiplicity along each component is then  computed by restricting to test surfaces, and by studying the push-forward via the map $p\colon \BM_{3,1}\rightarrow \BM_3$. The classes of $\overline{\mathcal{F}}_{3,1}$ and $\overline{\mathcal{H}}^+_4$  follow from the computation of the classes of the other components in each intersection. Throughout the paper, we provide various checks on Theorems \ref{res31} and \ref{h4+} (see \S \ref{checkM31}, (\ref{kappa2M4}), Remark \ref{DR2}, \S \ref{proofF}, (\ref{checkH4+})).

We have not succeeded in applying the above ideas to obtain a complete formula for the class of $\overline{\mathcal{H}}^-_4$. One can directly intersect the locus $\overline{\mathcal{H}}^-_4$ with some test surfaces and thus obtain linear relations on the coefficients of the class. While in this way we have produced some relations for such coefficients, at the moment we have not succeeded in computing the whole class. Nevertheless, in \S \ref{det-} we  compute the class of ${\mathcal{H}}^-_4$ in $A^2(\MM_4)$ (that is, the coefficient of $\lambda^2$) using a determinantal description.

\vskip4pt

\noindent {\bf Notation.}
We use throughout the following notation for divisor classes in ${\rm Pic}(\BM_{g,n})$.
For $i=1,\dots,n$, let $\psi_i$ be the cotangent line bundle class at the $i$-th marked point. 
Let $\lambda$ be the first Chern class of the Hodge bundle, and $\delta_0$ be the class of the locus $\Delta_0$ whose general element is a nodal irreducible curve. For $i=0,\dots,g$ and $S\subseteq \{1,\dots,n\}$, let $\delta_{i,S}$ be the class of the locus $\Delta_{i,S}$ whose general element has a component of genus $i$ containing the points with markings in $S$, and meeting transversally in one point a component of genus $g-i$ containing the remaining marked points. We write $\delta_i:= \delta_{i,\emptyset}$ and $\Delta_i:= \Delta_{i,\emptyset}$. When $n=1$, we write $\psi:= \psi_1$, $\delta_{i,1}:=\delta_{i,\{1\}}$, and $\Delta_{i,1}:=\Delta_{i,\{1\}}$.

We  also use the following codimension-two classes. Let $\kappa_2:= (p_{n+1})_*((\psi_{n+1})^3)$, where 
\linebreak $p_{n+1} \colon \BM_{g,n+1} \rightarrow \BM_{g,n}$ is the natural map obtained by forgetting the last marked point.
Let $\delta_{00}$ be the class of the closure of the locus of irreducible curves with two nodes. Let $\gamma_{i,S}$ be the class of the closure of the locus $\Gamma_{i,S}$ of curves with a component of genus $i$ containing the points with markings in $S$ and meeting in two points a component of genus $g-i-1$ containing the remaining marked points. We write $\gamma_i:=\gamma_{i,\emptyset}$ and $\Gamma_i:=\Gamma_{i,\emptyset}$.
Let $\delta_{01a}$ be the class of the closure of the locus of curves with a rational nodal tail. Finally, $\delta_{1|1}$ is the class of the closure of the locus of curves with two unmarked elliptic tails.

%$\delta_{11|}$ denotes the class of the locus of curves with two unmarked elliptic tails, either in $A^2(\BM_{3,1})$, or in $A^2(\BM_{2,2})$. 

We work over an algebraically closed field of characteristic $0$. All cycle classes are stack fundamental classes, and all cohomology and Chow groups are taken with rational coefficients. 

\vskip4pt

\noindent {\bf Acknowledgements.} We would like to thank Izzet Coskun, Joe Harris, Scott Mullane, and Anand Patel for helpful conversations 
on related topics. 
We are grateful to the referee for many valuable suggestions for improving the exposition of the paper.

%\setcounter{tocdepth}{1}
%\tableofcontents

\section{Enumerative geometry of general curves}
\label{ingre}

In this section, we collect some results about the enumerative geometry of general curves for later use.

\subsection{The difference map}
Let $C$ be a general curve of genus two. Define a generalized Abel-Jacobi map $f_d\colon C\times C\to \Pic^1(C)$ as follows: 
\[
(x, y) \mapsto \OO_{C}((d+1)x - dy), 
\]
where $d\in \bbZ^{+}$. 
Let $\Delta_{C\times C}\subset C\times C$ be the diagonal and  $I\subset C\times C$ be the locus of pairs of  
points that are conjugate under the hyperelliptic involution. The intersection of $I$ and $\Delta_{C\times C}$ consists of the six Weierstrass points of $C\cong \Delta_{C\times C}$.

\begin{prop}
\label{diff}
Under the above setting, we have:
\begin{samepage}
\begin{enumerate}
\item  $f_d$ is finite of degree $2 d^2 ( d + 1)^2$; 
\item  $f_d$ is simply ramified along $\Delta_{C\times C}$ and $I$, away from $I\cap \Delta_{C\times C}$;
\item  The ramification order of $f_d$ is $2$ at each point in the intersection $I\cap \Delta_{C\times C}$.  
\end{enumerate}
\end{samepage}
\end{prop}

\begin{proof}
Take a general point $p\in C$. Consider the isomorphism $u\colon \Pic^1(C)\to J(C)$ given by 
$
L\mapsto L\otimes \OO_C(-p). 
$
 Let $h_d = u \circ f_d$. Note that 
$
h_d(x, y) = \OO_C( (d+1)(x-p) - d(y-p)).
$
Let $\theta$ be the fundamental class of the theta divisor in $J(C)$. 
%It is well known that 
%\[
%\deg \theta^2 = g! = 2, 
%\]
%see \cite[\S 1.5]{MR770932}. Moreover, f
For $k\in\bbZ$, the locus of $\OO_C(k (x-p))$ for varying $x\in C$ has class $k^2 \theta$ in $J(C)$. Therefore, we conclude that 
\[
\deg f_d = \deg h_d  =  \deg h_{d*} h_d^{*}([\OO_C])  =  \deg \left((d+1)^2 \theta  \cdot d^2 \theta\right)   =  2 d^2 ( d + 1)^2, 
\]
thus proving the degree part of the proposition. 

Next, let $\phi\colon C\to \bbP^1$ be the hyperelliptic double covering. By \cite[p.~262]{MR770932}, the associated 
(projectivized) tangent space map of $f_d$ at $(x, y)$ can be regarded as $\bbP^1 \to \overline{\langle\phi(x), \phi(y)}\rangle$, 
where $\overline{\langle\phi(x), \phi(y)\rangle}$ is the linear span of $\phi(x)$ and $\phi(y)$ in $\bbP^1$. Therefore, $f_d$ is ramified at $(x,y)$ if and only if 
$\phi(x) = \phi(y)$, that is, $(x, y)\in \Delta_{C\times C} \cup I$. In particular, $f_d$ is finite away from $\Delta_{C\times C}$ and $I$. Moreover,  $f_d(w, w) = \OO_C(w)$ for any Weierstrass point $w\in C$, hence $f_d(\Delta_{C\times C})$ and $f_{d}(I)$ cannot be a single point. We thus conclude that $f_d$ is finite. 
\begin{comment}%%%%%%%%%%%%%%%%%%%%%%%%%%%%%%%
To compute the ramification orders along $\Delta_{C\times C}$ and $I$, we  further describe $h_d$ by local coordinates as follows. 
In a neighborhood of a Weierstrass point $q$, we can parameterize $C$ as $w^2 = z$ with respect to $q= (0,0)$. A basis of $H^0(C, K)$ is given by $\omega_1 = dw$ and $\omega_2 = w^2 dw$ locally around $q$. Note that $J(C)\cong H^0(C, K)^{*}/H_1(C, \bbZ)$, hence $h_d$ is given by 
\[
(x, y)\mapsto \left( (d+1) \int_{p}^x \omega_1 - d \int_{p}^y \omega_1, \ (d+1) \int_{p}^x \omega_2 - d \int_{p}^y \omega_2\right)
\]
modulo $H_1(C, \bbZ)$.
Therefore, we see that $h_d$ is locally analytically given by 
\[
(s, t) \mapsto \left((d+1)s - dt,\  \frac{1}{3}(d+1)s^3 - \frac{1}{3}d t^3\right), 
\]
where  $s, t$ are the $w$-coordinates of $x, y \in C$, respectively. Note that $\Delta_{C\times C}$ and $I$ correspond to the loci $s-t = 0$ and $s+t = 0$, respectively. One easily checks that in a neighborhood of $(0,0)$,  
the ramification orders of $h_d$  at $\Delta_{C\times C} \backslash (0,0)$, $I \backslash (0,0)$, and $(0,0)$ are $1$, $1$, and $2$, respectively. Since $h_d = u \circ f_d$ and 
$u$ is an isomorphism, this thus proves the remaining part of the proposition. 
\end{comment}%%%%%%%%%%%%%%%%%%%%%%%%%%%%%%%%%
Finally, via a local computation in a neighborhood of a Weierstrass point of $C$, one verifies the ramification orders along $\Delta_{C\times C}$, $I$, and $\Delta_{C\times C}\cap I$.
\end{proof}

Similarly, we consider the following map. Let $p$ be a fixed point on a general curve $C$ of genus two.  
Fix two positive integers $d_1, d_2$, and let $d =d_1 - d_2 - 1$. Define the map 
$f_{d_1, d_2}\colon C\times C\rightarrow {\rm Pic}^{d}(C)$ by 
\[
(x, y) \mapsto \mathcal{O}_C(d_1 x - d_2 y - p).
\]
The same proof as in Proposition~\ref{diff} implies the following result. 
\begin{prop}
\label{diff2}
The degree of $f_{d_1, d_2}$ is $2 d_1^2 d_2^2$.
\end{prop}

\subsection{The Scorza curve}
\label{Scorza}

Let $C$ be a curve of genus $g$, and $\eta^+$ an even theta characteristic on $C$. Suppose that $\eta^+$ is not a vanishing theta-null, that is, $h^0(C, \eta^+)=0$. The Scorza curve $T_{\eta^+}$ in $C\times C$ is defined as
\[
T_{\eta^+} := \{(x,y)\in C\times C \,| \, h^0(\eta^+\otimes \mathcal{O}(x-y))\neq 0 \}
\]
(see \cite{Scorza}). %, or \cite[\S 7]{MR1213725} for a modern treatment.)
From the assumption, $T_{\eta^+}$ does not meet the diagonal $\Delta_{C\times C}\subset C\times C$. By the Riemann-Roch theorem, $T_{\eta^+}$ is a symmetric correspondence.

We  need the following computation: when $T_{\eta^+}$ is reduced, its class in $H^2(C\times C)$ is
\[
\left[T_{\eta^+}\right] = (g-1) F_1 + (g-1) F_2 + \Delta_{C\times C},
\]
where $F_1$ and $F_2$ are the classes of a fiber of the projections $C\times C \rightarrow C$ of the first and second components, respectively (see \cite[\S 7.1]{MR1213725}).
We  also use that for a general spin curve $[C,\eta^+]\in \mathcal{S}_g^+$ with $g\geq 2$, the Scorza curve $T_{\eta^+}$ is smooth. In particular, the locus of pairs $(x,y)\in C\times C$ such that
\[
h^0(C, \eta^+\otimes\mathcal{O}(x-2y))\geq 1 \quad \mbox{and}\quad h^0(C,\eta^+\otimes\mathcal{O}(y-2x))\geq 1 
\]
is empty (see for instance \cite[Theorem 4.1]{MR3245010}).

\section{Tautological classes on \texorpdfstring{$\BM_{3,1}$}{BM31}}
\label{basisforMbar31}

In this section, we show that the classes in (\ref{bm31gen}) form a basis of $R^2(\BM_{3,1})$, and thus prove Proposition \ref{basisR2Mbar31}. 
Moreover, in \S \ref{secboundaryinA2Mbar31} we express certain boundary strata classes in terms of such a basis.
We  use these results in \S \ref{pullbacktoMbar31} and \S \ref{SF}.

\subsection{Proof of Proposition \ref{basisR2Mbar31}}
From \cite{getzler-looijenga} we know that the dimension of $H^4(\BM_{3,1})$ is $16$ (see also \cite{MR2433615}). Since the intersection pairing of classes in the tautological group $RH^2(\BM_{3,1})\subseteq H^4(\BM_{3,1})$ with tautological classes of complementary dimension has rank $16$ (see \cite{yang}), we conclude that $RH^2(\BM_{3,1})=H^4(\BM_{3,1})$. Every relation among codimension-two tautological classes in $RH^2(\BM_{3,1})$ lifts via the cycle map to $R^2(\BM_{3,1})$, hence $R^2(\BM_{3,1})\simeq RH^2(\BM_{3,1})=H^4(\BM_{3,1})$.  

In order to prove that the classes in (\ref{bm31gen}) form a basis, it is enough to show that the intersection pairing of the classes in (\ref{bm31gen}) with classes in complementary dimension has rank $16$. This is the strategy used by Getzler for $H^*(\BM_{2,2})$ in \cite{MR1672112}.

Here we use a different approach, relying on Getzler's results. Remember that $R^2(\BM_{2,2})$ has dimension $14$, and one has $R^2(\BM_{2,2})\simeq RH^2(\BM_{2,2})=H^4(\BM_{2,2})$. The idea is the following. First we consider the natural map $\vartheta\colon \BM_{2,2}\rightarrow \BM_{3,1}$ obtained by attaching a fixed elliptic tail to the second marked point. We show that the pull-backs of the classes in (\ref{bm31gen}) via $\vartheta$ span a $13$-dimensional subspace of $R^2(\BM_{2,2})$. Whence we compute the possible three-dimensional space of relations which may still hold among the above classes, depending on three parameters $\alpha, \beta, \gamma$. Finally, using test surfaces we conclude that $\alpha=\beta=\gamma=0$.

A basis for the Picard group of $\BM_{2,2}$ is given by the following divisor classes
\begin{align*}
\psi_1, \quad \psi_2, \quad \delta_{0,\{1,2\}}, \quad \delta_0, \quad \delta_{1,\{1\}}, \quad \delta_{1,\{1,2\}}.
\end{align*}
Getzler shows that the following relations hold among products of divisor classes
\begin{align*}
 \delta_{1,\{1,2\}}\left(12\delta_{1,\{1\}}+12\delta_{1,\{1,2\}}+\delta_0 \right)=\delta_{1,\{1\}}\left(12\delta_{1,\{1\}}+12\delta_{1,\{1,2\}}+\delta_0 \right)=0,\nonumber\\
\delta_{1,\{1\}}\left(\psi_1+\psi_2+\delta_{1,\{1\}} \right)=\psi_1\delta_{0,\{1,2\}}=\psi_2\delta_{0,\{1,2\}}=\delta_{1,\{1\}}\delta_{0,\{1,2\}}=0,\\
\left(\psi_1-\psi_2 \right)\left(10\psi_1+10\psi_2-2\delta_{1,\{1\}}-12\delta_{1,\{1,2\}}-\delta_0 \right)=0,\nonumber
\end{align*}
hence a basis for $R^2(\BM_{2,2})$ is given by the following products
\begin{align*}
\psi_1\psi_2, \quad \psi_2^2,  \quad \psi_1\delta_{1,\{1\}},  \quad \psi_2\delta_{1,\{1\}}, \quad 
 \psi_1\delta_{1,\{1,2\}},  \quad \psi_2\delta_{1,\{1,2\}} , \quad \psi_1\delta_0,\quad \psi_2\delta_0,\\
\delta_{0,\{1,2\}}^2 , \quad  \delta_{1,\{1,2\}}\delta_{0,\{1,2\}} , \quad \delta_0\delta_{0,\{1,2\}},  \quad \delta_0\delta_{1,\{1\}},\quad 
 \delta_0\delta_{1,\{1,2\}}, \quad \delta_0^2.
\end{align*}
The following formulae are well-known
\begin{align*}
\vartheta^*(\psi) &= \psi_1, & \vartheta^*(\lambda) &= \lambda=\frac{1}{10}\delta_0+\frac{1}{5}(\delta_{1,\{1\}}+\delta_{1,\{1,2\}}), &
\vartheta^*(\delta_0) &= \delta_0, \\
 \vartheta^*(\delta_{1,1}) &= \delta_{1,\{1\}} + \delta_{0,\{1,2\}}, & \vartheta^*(\delta_{2,1}) &= -\psi_2 + \delta_{1,\{1,2\}}. 
\end{align*}
Moreover, from the computations in \cite{MR1672112} it follows that
\begin{eqnarray}
\label{delta01|12}
\vartheta^*(\delta_{01a}) &=& \delta_{01a}\nonumber\\
&=& 24\psi_2^2 -6\psi_1\delta_{1,\{1\}} -\frac{54}{5}\psi_2\delta_{1,\{1\}}+ 6\psi_1\delta_{1,\{1,2\}} -\frac{114}{5}\psi_2\delta_{1,\{1,2\}} -\frac{12}{5}\psi_2\delta_0\\
&&{}+24\delta_{0,\{1,2\}}^2 +\frac{84}{5}\delta_{1,\{1,2\}}\delta_{0,\{1,2\}} + \frac{12}{5}\delta_0\delta_{0,\{1,2\}} + \frac{47}{50}\delta_0\delta_{1,\{1\}}+ \frac{36}{25}\delta_0\delta_{1,\{1,2\}}+ \frac{3}{25}\delta_0^2.\nonumber
\end{eqnarray}
%Following the notation in \cite{MR1672112}, $\delta_{01|12}$ denotes the class of the locus of curves with an unmarked rational nodal tail.
Getzler expresses the basis of $R^2(\BM_{2,2})$ given by products of divisor classes in terms of decorated boundary strata classes. By the inverse change of basis, we obtain the formula (\ref{delta01|12}).

It is easy to show that $\vartheta^*(\kappa_2)=\kappa_2$. Note that by Mumford's relation, one has
\[
\kappa_2 = \lambda (\lambda+\delta_1) \in A^2(\BM_2).
\]
Using the well-known formula $\kappa_i = p_n^*(\kappa_i) +\psi_n^i$, where $p_n\colon \BM_{g,n}\rightarrow \BM_{g,n-1}$ is the natural map,
we obtain that
\begin{eqnarray}
\label{kappa2Mbar22}
\vartheta^*(\kappa_2)=\kappa_2 &=& \lambda(\lambda+\delta_{1,\{1\}}+\delta_{1,\{1,2\}}) +\psi_1^2 + \psi_2^2 +\delta_{0,\{1,2\}}^2  \\
&=& \psi_1^2 + \psi_2^2 +\delta_{0,\{1,2\}}^2 +\frac{3}{25} \delta_0(\delta_{1,\{1\}}+\delta_{1,\{1,2\}}) +\frac{1}{100} \delta_0^2 \in A^2(\BM_{2,2}). \nonumber
\end{eqnarray}

Using the above formulae, it is easy to compute that the pull-back via $\vartheta$ of the classes in the statement span a subspace of $R^2(\BM_{2,2})$ of dimension $13$. The only possible relations are given by
\begin{eqnarray*}
\label{posrel}
\alpha\cdot \left({}-2\psi^2+6\psi\lambda-\frac{1}{2}\psi\delta_0 -2\psi\delta_{1,1} -6\lambda^2 +\frac{1}{2}\lambda\delta_0+6\lambda\delta_{2,1}
-\frac{1}{2}\delta_0\delta_{2,1} -\delta_{1,1}^2 +\kappa_2 \right)\nonumber\\
+\beta \cdot \left({}-\psi^2 +6\psi\lambda -\frac{1}{2}\psi\delta_0 -5\lambda^2 +\frac{1}{2}\lambda\delta_0 -4\lambda\delta_{2,1} 
+\frac{1}{2} \delta_0\delta_{2,1} + \delta_{2,1}^2 \right)\\
+\gamma \cdot\left( 120\lambda^2-22\lambda\delta_0 +\delta_0^2 \right)=0 \nonumber
\end{eqnarray*}
for $\alpha, \beta, \gamma \in \mathbb{Q}$.
Restricting this relation to the three test surfaces in \S\ref{TSM31} yields $\alpha=\beta=\gamma=0$. 
This completes the proof of the statement.
\hfill $\square$

\subsection{Three test surfaces in \texorpdfstring{$\BM_{3,1}$}{BM{3,1}}}
\label{TSM31}
In the next subsections, we  consider the intersections of the classes in Proposition \ref{basisR2Mbar31} with  three test surfaces. These computations are used in the proof of Proposition \ref{basisR2Mbar31}.
Note that $\kappa_2$ has zero intersection with the two-dimensional families of curves whose base is a product of two curves.

\subsubsection{}
\label{7} 
Let $(E_1, p, q_1)$ be a general two-pointed elliptic curve, and let $(E_2, q_2)$ be a pointed elliptic curve. Consider the surface in $\BM_{3,1}$ obtained by identifying the points $q_1, q_2$, and by identifying a moving point in $E_1$ with a moving point in $E_2$.

\begin{figure}[htbp]
\centering
     \includegraphics[scale=0.7]{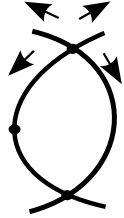}
  \caption{How the general fiber of the family in \S \ref{7} moves.}
\end{figure}

The base of the surface is $E_1 \times E_2=:S_1$. Let $\pi_i\colon E_1 \times E_2 \rightarrow E_i$ be the natural projection, for $i=1,2$. The divisor classes restrict as follows
\begin{align*}
\psi|_{S_1} &= \pi_1^*[p], &  \lambda|_{S_1} &= 0, & \delta_0|_{S_1} &= -2\pi_1^*[q_1] -2\pi_2^*[q_2], &
\delta_{1,1}|_{S_1} &= \pi_1^*[q_1], & \delta_{2,1}|_{S_1} &= \pi_2^*[q_2]. 
\end{align*}
The classes in Proposition \ref{basisR2Mbar31} with nonzero intersection with this test surface are thus the following
\begin{align*}
\psi\delta_0|_{S_1} &= -2, & \psi\delta_{2,1}|_{S_1} &= 1, & \delta_0^2|_{S_1} &= 8, &
\delta_0\delta_{1,1}|_{S_1} &= -2, & \delta_0\delta_{2,1}|_{S_1} &= -2, & \delta_{1,1} \delta_{2,1} |_{S_1}&= 1.
\end{align*}

\subsubsection{}
\label{10} 
Let $(E_1, q_1, r)$ and $(E_2, q_2, p)$ be two general two-pointed elliptic curves. Identify the points $q_1, q_2$; identify the point $r$ with a moving point in $E_1$; finally move the curve $E_2$ in a pencil of degree $12$.

\begin{figure}[htbp]
\centering
 \includegraphics[scale=0.7]{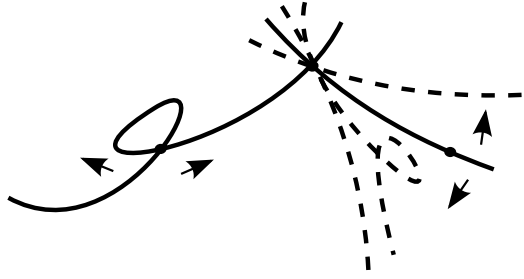}
  \caption{How the general fiber of the family in \S \ref{10} moves.}
\end{figure}

The base of this surface is $E_1\times \mathbb{P}^1=:S_2$. Let $x$ be the class of a point in $\mathbb{P}^1$. The divisor classes restrict as follows
\begin{align*}
\psi |_{S_2} &= \pi_2^*(x) = \lambda |_{S_2}, & \delta_0 |_{S_2} &= -2\pi_1^*[r] +12\pi_2^*(x), &
\delta_{1,1} |_{S_2} &= -\pi_1^*[q_1] -\pi_2^*(x), & \delta_{2,1} |_{S_2} &= \pi_1^*[r], 
\end{align*}
and the non-zero restrictions of the generating codimension-two classes are thus
\begin{align*}
\psi\delta_0 |_{S_2} &= -2, & \psi\delta_{1,1} |_{S_2} &= -1, & \psi\delta_{2,1} |_{S_2} &= 1, & \lambda \delta_0 |_{S_2} &= -2, \\
\lambda\delta_{2,1} |_{S_2} &= 1, & \delta_0^2 |_{S_2} &= -48, & \delta_0\delta_{1,1} |_{S_2} &= -10, & \delta_0\delta_{2,1}|_{S_2} &= 12, \\
 \delta_{1,1}^2 |_{S_2} &= 2, & \delta_{1,1} \delta_{2,1} |_{S_2} &= -1. &  &  
 \end{align*}

\subsubsection{} 
\label{11}
Let $(E,p,q,r)$ be a general three-pointed elliptic curve. Consider the surface obtained by identifying the point $r$ with a moving point in $E$, and by attaching at the point $q$ an elliptic tail moving in a pencil of degree $12$.

\begin{figure}[htbp]
\centering
 \includegraphics[scale=0.7]{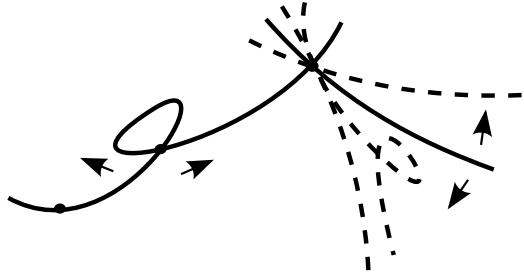}
  \caption{How the general fiber of the family in \S \ref{11} moves.}
\end{figure}

The base of this surface is $E\times \mathbb{P}^1=:S_3$. The divisor classes restrict as follows
\begin{align*}
\psi |_{S_3} &= \pi_1^*[p], &  \lambda |_{S_3} &= \pi_2^*(x), &
\delta_0 |_{S_3} &= -2\pi_1^*[r] +12\pi_2^*(x), & \delta_{2,1} |_{S_3} &= -\pi_2^*(x). 
\end{align*}
We deduce that the non-zero restrictions of product of divisor classes are the following
\begin{align*}
\psi\lambda |_{S_3} &= 1, & \psi\delta_0 |_{S_3} &= 12, & \psi\delta_{2,1} |_{S_3} &= -1, & \lambda\delta_0 |_{S_3} &= -2, & \delta_0^2 |_{S_3} &= -48, & \delta_0\delta_{2,1} |_{S_3} &=2.
\end{align*}
Moreover, we have $\delta_{01a}|_{S_3} = -\pi_1^*[q] \cdot 12\pi_2^*(x)= -12$.

\subsection{Some equalities in \texorpdfstring{$A^2(\BM_{3,1})$}{A2BM31}} 
\label{secboundaryinA2Mbar31}
In \S \ref{pullbacktoMbar31} we  also use the following result. 

\begin{prop}
\label{boundaryinA2Mbar31}
The following equalities hold in $A^2(\BM_{3,1})$
\begin{eqnarray*}
\lambda\delta_{1,1} &=& \frac{1}{5}\left({}-\psi+\frac{1}{2}\delta_0 +\delta_{2,1} \right)\delta_{1,1},\\
\delta_{00} &=& -12\psi^2 -24\psi\delta_{1,1}-372\lambda^2 +72\lambda\delta_0 +120\lambda\delta_{2,1} -3\delta_0^2 -12\delta_0\delta_{2,1}\\
	&&{}-12\delta_{1,1}^2 -12\delta_{2,1}^2 +12\kappa_2,\\
\delta_{1|1} &=& \frac{1}{2}\psi^2 +\frac{1}{5}\psi\delta_{1,1} +5\lambda^2 -\frac{1}{2}\lambda\delta_0 +4\lambda\delta_{2,1} -\frac{1}{10}\delta_0\delta_{1,1} -\frac{1}{2}\delta_0\delta_{2,1}\\
	&&{}-\frac{1}{2} \delta_{1,1}^2 -\frac{6}{5}\delta_{1,1}\delta_{2,1} -\frac{1}{2}\delta_{2,1}^2 +\frac{1}{12}\delta_{01a} -\frac{1}{2}\kappa_2,\\
\gamma_1 &=& \frac{78}{5}\psi\delta_{1,1} +126\lambda^2 -\frac{55}{2}\lambda\delta_0 -78\lambda\delta_{2,1} +\frac{3}{2}\delta_0^2 +\frac{7}{10}\delta_0\delta_{1,1} +\frac{17}{2}\delta_0\delta_{2,1}\\
	&&{}+12\delta_{1,1}^2 +\frac{42}{5}\delta_{1,1}\delta_{2,1} +12\delta_{2,1}^2,\\
\gamma_2 &=& \frac{15}{2}\psi^2 -21\psi\lambda +2\psi\delta_0 +\psi\delta_{1,1} +3\psi\delta_{2,1} +\frac{101}{2}\lambda^2 -10\lambda\delta_0 -19\lambda\delta_{2,1}\\
	&&{} +\frac{1}{2}\delta_0^2 +2\delta_0\delta_{2,1} +\frac{1}{2}\delta_{1,1}^2  +\frac{5}{2}\delta_{2,1}^2 -\frac{1}{2}\kappa_2.
\end{eqnarray*}
\end{prop}

\begin{proof}
The idea is to follow the lines of the proof of Proposition \ref{basisR2Mbar31}. Let us start with the second equality. 
Write $\delta_{00}$ as a linear combination of the classes in (\ref{bm31gen}) with unknown coefficients. By pulling-back via the map $\vartheta\colon \BM_{2,2}\rightarrow \BM_{3,1}$ and using $\vartheta^*(\delta_{00})=\delta_{00}$ and Mumford's equality
\[
\delta_{00} = 6\lambda\delta_0 \in R^2(\BM_{2,2}),
\]
we are able to determine the coefficients in the statement up to three variables $\alpha, \beta, \gamma$:
\begin{eqnarray*}
\delta_{00} &=& (-2 \alpha -\beta)\psi^2 +6(\alpha + \beta)\psi\lambda -\frac{\alpha+\beta}{2} \psi\delta_0 -2 \alpha\psi\delta_{1,1}\\
&&{} +(-6 \alpha - 5 \beta + 120 \gamma)\lambda^2 +(6 + \frac{\alpha + \beta}{2} - 22 \gamma)\lambda\delta_0 +(6 \alpha - 4 \beta)\lambda\delta_{2,1}\\
&&{}+  \gamma \delta_0^2 +\frac{\beta-\alpha}{2}\delta_0\delta_{2,1}  -\alpha \delta_{1,1}^2 + \beta \delta_{2,1}^2, +\alpha \kappa_2 \in R^2(\BM_{3,1}).
\end{eqnarray*}
Next, we restrict the above equality to the three test surfaces considered in \S\ref{TSM31}. It is easy to show that the restriction of $\delta_{00}$ to the first test surface is $0$, while is respectively
\[
\pi_1^*(-2[r]-[q_1])\cdot\pi_2^*(12x) + \pi_1^*[q_1]\cdot \pi_2^*(12x)=-24
\]
and
\[
\pi_1^*(-2[r]-[p]-[q])\cdot\pi_2^*(12x) + \pi_1^*[p]\cdot\pi_2^*(12x)+ \pi_1^*[q]\cdot\pi_2^*(12x)=-24
\]
to the other two test surfaces. It follows that $\alpha=12=-\beta$ and $\gamma= -3$, thus proving the equality for $\delta_{00}$. 

The first equality is proven in a similar way. Alternatively, one obtains it as the push-forward of divisor relations via the boundary map $\BM_{1,2}\times \BM_{2,1}\rightarrow \BM_{3,1}$.

For the third equality, note that
\[
\vartheta^*(\delta_{1|1}) = -\psi_2\delta_{1,\{1,2\}} + \delta_{1|1} \in R^2(\BM_{2,2})
\]
and by Getzler's computations
\begin{eqnarray*}
\delta_{1|1} &=&  \psi_2^2 -\frac{1}{2}\psi_1\delta_{1,\{1\}} -\frac{7}{10}\psi_2\delta_{1,\{1\}}+ \frac{1}{2}\psi_1\delta_{1,\{1,2\}} -\frac{7}{10}\psi_2\delta_{1,\{1,2\}} -\frac{1}{10}\psi_2\delta_0 \\
&&{} + \delta_{0,\{1,2\}}^2 + \frac{1}{5}\delta_{1,\{1,2\}}\delta_{0,\{1,2\}}+ \frac{1}{10}\delta_0\delta_{0,\{1,2\}}+ \frac{3}{50}\delta_0\delta_{1,\{1\}} +\frac{11}{600}\delta_0\delta_{1,\{1,2\}}+ \frac{1}{200}\delta_0^2
\end{eqnarray*}
in $R^2(\BM_{2,2})$.
Finally, the restriction of $\delta_{1|1}\in R^2(\BM_{3,1})$ to the third test surface in \S \ref{TSM31} is 
\[
\pi_1^*[r] (-\pi_2^*(x))=-1,
\]
while the restriction to the other two test surfaces is zero.

For $\gamma_1$ and $\gamma_2$, we use
\begin{align*}
\vartheta^*(\gamma_1) &=  \gamma_1 + \gamma_{1,\{1\}}, & \vartheta^*(\gamma_2) &= \gamma_{1,\{2\}}.
\end{align*}
%Above, we use the notation from \cite{MR1672112}: $\delta_{0|}$ is the class of the locus of curves with a rational component with the two marked points, meeting in two points an elliptic component; $\delta_{0|i}$ is the class of the locus of curves with a rational and elliptic components meeting in two points, with the point $i$ on the elliptic component, and the other on the rational component, for $i=1,2$.
By Getzler's computations, we have
\begin{align*}
\gamma_1 &= 3\psi_1\psi_2 +3\psi_2^2 -\frac{6}{5}\psi_1\delta_{1,\{1\}} -\frac{9}{5}\psi_2\delta_{1,\{1\}} -\frac{6}{5}\psi_1\delta_{1,\{1,2\}} -\frac{24}{5}\psi_2\delta_{1,\{1,2\}}  -\frac{1}{10}\psi_1\delta_0\\
	&{} -\frac{2}{5}\psi_2\delta_0 +12\delta_{0,\{1,2\}}^2 +\frac{42}{5}\delta_{1,\{1,2\}}\delta_{0,\{1,2\}} +\frac{7}{10}\delta_0\delta_{0,\{1,2\}} +\frac{3}{25}\delta_0(\delta_{1,\{1\}} +\delta_{1,\{1,2\}}) +\frac{1}{100}\delta_0^2,\\
\gamma_{1,\{1\}} &= 9\psi_2^2-3\psi_1\psi_2 +\frac{6}{5}\psi_1\delta_{1,\{1\}} -\frac{33}{5}\psi_2\delta_{1,\{1\}} +\frac{6}{5}\psi_1\delta_{1,\{1,2\}} -\frac{18}{5}\psi_2\delta_{1,\{1,2\}}
 +\frac{1}{10}\psi_1\delta_0 -\frac{3}{10}\psi_2\delta_0, \\
\gamma_{1,\{2\}} &= 9\psi^2_2 -3\psi_1\psi_2  -\frac{24}{5}\psi_1\delta_{1,\{1\}} -\frac{3}{5}\psi_2\delta_{1,\{1\}} +\frac{36}{5}\psi_1\delta_{1,\{1,2\}} -\frac{48}{5}\psi_2\delta_{1,\{1,2\}} 
+\frac{3}{5}\psi_1\delta_0 -\frac{4}{5}\psi_2\delta_0
\end{align*}
in $R^2(\BM_{2,2})$.
The intersection of $\gamma_1 \in R^2(\BM_{3,1})$ with the first test surface in \S \ref{TSM31} is 
\[
(-\pi_1^*([p]+[q_1])-\pi_2^*[q_2])\cdot(-\pi_1^*[q_1]-\pi_2^*[q_2])+\pi_1^*[p](-\pi_2^*[q_2])=2,
\]
while it is zero with the other two test surfaces.
Similarly, the intersection of $\gamma_2\in R^2(\BM_{3,1})$ with the first test surface in \S \ref{TSM31} is 
\[
\pi_1^*[p] (-\pi_2^*[q_2])=-1,
\]
while it is zero with the other two test surfaces.
\end{proof}

\subsection{A check}
\label{checkM31}
As a partial check,  we compute the expressions for $\delta_{00}$ and $\gamma_1$ in an alternative way. The classes $\delta_{00}$  and $\gamma_1$ in $R^2(\BM_{3,1})$ coincide with the pull-back via $p\colon \BM_{3,1}\rightarrow \BM_3$ of the classes $\delta_{00}$ and $\gamma_1$ in $A^2(\BM_3)$. Using the following identities in $A^2(\BM_3)$ 
\begin{eqnarray*}
\delta_{00} &=& {}-372 \lambda^2 +72 \lambda\delta_0 +120\lambda\delta_{1} -3\delta_0^2 -12\delta_0\delta_{1} -12\delta_{1}^2 +12\kappa_2,\\
\gamma_1 &=& 126 \lambda^2 -\frac{55}{2}\lambda\delta_0 -78\lambda\delta_1 +\frac{3}{2}\delta_0^2 +\frac{17}{2}\delta_0\delta_1 +12\delta_1^2
\end{eqnarray*}
from \cite[Theorem 2.10]{MR1070600}, and the following well-known identities
\begin{align*}
p^*(\lambda) &= \lambda, & p^*(\delta_0) &= \delta_0, & p^*(\delta_1) &= \delta_{1,1}+\delta_{2,1}, & \kappa_2 &= p^*(\kappa_2) + \psi^2,
\end{align*}
one recovers the two formulae in Proposition \ref{boundaryinA2Mbar31}.

\section{The locus of Weierstrass points on hyperelliptic curves in \texorpdfstring{$\overline{\mathcal{M}}_{3,1}$}{BM31}}
\label{pullbacktoMbar31}

In this section we compute the class of the closure in $\BM_{3,1}$ of the locus of Weierstrass points on hyperelliptic curves of genus $3$
\[
 \mathcal{H}yp_{3,1} := \{[C,p]\in \MM_{3,1} \,|\, C \,\,\mbox{is hyperelliptic and}\,\, p \,\,\mbox{is a Weierstrass point in} \,\, C  \}.
\]
The idea is to consider the pull-back of the hyperelliptic locus in $\BM_4$ via the clutching map $j_3\colon \overline{\mathcal{M}}_{3,1}\rightarrow \BM_4$ obtained by attaching a fixed elliptic tail at the marked point of an element $[C,p]\in \BM_{3,1}$. Using the theory of admissible covers, the following statement is straightforward.

\begin{lem}
\label{j3^*H}
We have the following equality in $A^2(\BM_{3,1})$
\begin{eqnarray*}
j_{3}^*\left(\overline{\mathcal{H}yp}_4\right) &=& \overline{\mathcal{H}yp}_{3,1}.
\end{eqnarray*}
\end{lem}

Following \cite{MR1078265}, we use the basis of $A^2(\BM_4)$ given by the classes
\begin{eqnarray} 
\label{basisA2M4}
\lambda^2, \lambda\delta_0, \lambda\delta_1, \lambda\delta_2, \delta_0^2, \delta_0\delta_1, \delta_1^2, 
\delta_1\delta_2, \delta_2^2, \delta_{00}, \delta_{01a}, \gamma_1, \delta_{1|1}.
\end{eqnarray}
The class $\left[\overline{\mathcal{H}yp}_4 \right]$ in terms of this basis has been computed in \cite[Proposition 5]{MR2120989}:
\begin{eqnarray}
\label{hyp4}
 \left[ \overline{\mathcal{H}yp}_{4} \right] & = & \frac{51}{4}\lambda^2 -\frac{31}{10}\lambda\delta_0 -\frac{117}{10}\lambda\delta_1 + 3\lambda\delta_2+ \frac{7}{40}\delta_0^2 + \frac{7}{5}\delta_0\delta_1\\
&& {} + \frac{21}{10}\delta_1^2 + 3\delta_1\delta_2 + \frac{9}{2}\delta_2^2+ \frac{1}{40}\delta_{00} -\frac{3}{10}\gamma_1 -\frac{3}{40}\delta_{01a}+ \frac{9}{10}\delta_{1|1}. \nonumber
\end{eqnarray}
Note that the above class agrees with the one in \cite{MR2120989} modulo the relation (\ref{kappa2M4}).

Let us compute the pull-back via $j_3\colon \BM_{3,1}\rightarrow \BM_4$ of the basis of $A^2(\BM_4)$. We use the notation from \S\ref{basisforMbar31} for classes in $R^*(\BM_{3,1})$. The pull-back of the basis of $\Pic(\BM_4)$ is as follows
\begin{align*}
j_3^*(\lambda) &= \lambda, & j_3^*(\delta_0) &= \delta_0, & j_3^*(\delta_1) &= -\psi+\delta_{2,1}, & j_3^*(\delta_2) &= \delta_{1,1}.
\end{align*}
Moreover, we have
\begin{eqnarray*}
j_3^*(\delta_{01a}) &=& \delta_{01a},\\
j_3^*(\delta_{00}) &=& \delta_{00}\\
	&=& -12\psi^2 -24\psi\delta_{1,1}-372\lambda^2 +72\lambda\delta_0 +120\lambda\delta_{2,1} -3\delta_0^2 -12\delta_0\delta_{2,1}\\
	&&{}-12\delta_{1,1}^2 -12\delta_{2,1}^2 +12\kappa_2,\\
j_3^*(\delta_{1|1}) &=& {}-\psi\delta_{2,1}+\delta_{1|1}\\
	&=&{} \frac{1}{2}\psi^2 +\frac{1}{5}\psi\delta_{1,1} -\psi\delta_{2,1}+5\lambda^2 -\frac{1}{2}\lambda\delta_0 +4\lambda\delta_{2,1} -\frac{1}{10}\delta_0\delta_{1,1}\\
	&&{} -\frac{1}{2}\delta_0\delta_{2,1} -\frac{1}{2} \delta_{1,1}^2 -\frac{6}{5}\delta_{1,1}\delta_{2,1} -\frac{1}{2}\delta_{2,1}^2 +\frac{1}{12}\delta_{01a} -\frac{1}{2}\kappa_2,\\
j_3^*(\gamma_1) &=& \gamma_1 + \gamma_2\\
	&=& \frac{15}{2}\psi^2 -21\psi\lambda +2\psi\delta_0 +\frac{83}{5}\psi\delta_{1,1} +3\psi\delta_{2,1} +\frac{353}{2}\lambda^2 -\frac{75}{2}\lambda\delta_0 -97\lambda\delta_{2,1}\\
		&&{} +2\delta_0^2 +\frac{7}{10}\delta_0\delta_{1,1} +\frac{21}{2}\delta_0\delta_{2,1} +\frac{25}{2}\delta_{1,1}^2 +\frac{42}{5}\delta_{1,1}\delta_{2,1}
 +\frac{29}{2}\delta_{2,1}^2 -\frac{1}{2}\kappa_2.
\end{eqnarray*}
For the above equalities we have used Proposition \ref{boundaryinA2Mbar31}.

As a partial check for the above formulae, remember that the following relation holds in $A^2(\BM_4)$
\begin{multline}
\label{kappa2M4}
60 \kappa_2 - 810 \lambda^2 + 156 \lambda\delta_0 + 252 \lambda \delta_1 -3\delta_0^2 - 24\delta_0\delta_1 
{}+24\delta_1^2 -9\delta_{00}+7\delta_{01a} -12\gamma_1-84\delta_{1|1}=0
\end{multline}
(see \cite{MR1078265}).
Using the identity $j_3^*(\kappa_2)=\kappa_2$, one can easily verify that the above formulae satisfy the pull-back via $j_3$ of this relation. 

As a consequence of Proposition \ref{j3^*H}, by pulling-back via $j_3$ the  class of the hyperelliptic locus $\overline{\mathcal{H}yp}_{4}$ in $\BM_{4}$ (\ref{hyp4}) and by means of the above pull-back formulae, we obtain the class of $\overline{\mathcal{H}yp}_{3,1}$.

\begin{thm}
\label{Hhyp31}
The class of $\overline{\mathcal{H}yp}_{3,1}$ in $A^2(\BM_{3,1})$ is
\[
\overline{\mathcal{H}yp}_{3,1} \equiv  \psi (18\lambda -2\delta_0 -9\delta_{1,1} -6\delta_{2,1})-\lambda\left(45\lambda -\frac{19}{2}\delta_0 -24\delta_{2,1}\right) -\frac{1}{2}\delta_0^2 -\frac{5}{2}\delta_0\delta_{2,1} -3\delta_{2,1}^2.
\]
\end{thm}

\begin{rem}
\label{DR2}
An alternative way to prove Theorem \ref{Hhyp31} is to follow the idea used in \S \ref{basisforMbar31}: the pull-back of $\overline{\mathcal{H}yp}_{3,1}$ via the natural map $\vartheta\colon \BM_{2,2}\rightarrow \BM_{3,1}$ is the closure of the locus $\mathcal{DR}_2(2)$ of curves of genus $2$ with two marked Weierstrass points. This locus is a special case of the double ramification locus, and its class in $A^2(\BM_{2,2})$ has been computed in \cite{DR}:
\[
\overline{\mathcal{DR}}_2(2) \equiv 6\psi_1\psi_2 -3\psi_2^2 -\frac{12}{5}\psi_1\delta_{1,\{1\}} -\frac{9}{5}\psi_2\delta_{1,\{1\}} -\frac{12}{5}\psi_1\delta_{1,\{1,2\}}+ \frac{6}{5}\psi_2\delta_{1,\{1,2\}} -\frac{1}{5}\psi_1\delta_0+ \frac{1}{10} \psi_2\delta_0.
\]
Imposing $\vartheta^*(\overline{\mathcal{H}yp}_{3,1})=\overline{\mathcal{DR}}_2(2)$ determines the class of $\overline{\mathcal{H}yp}_{3,1}$ up to three coefficients. Finally, one can find the values of such coefficients using three test surfaces.
\end{rem}

\section{The locus of marked hyperflexes in \texorpdfstring{$\BM_{3,1}$}{BM31}}
\label{SF}

In this section, we compute the class of the locus of genus-$3$ curves with a marked hyperflex point
\[
 \mathcal{F}_{3,1} := \{[C,p]\in \MM_{3,1} \,| \,\,\mbox{$C$ is not hyperelliptic and $p$ is a hyperflex point in $C$}  \}.
\]
Equivalently, $\mathcal{F}_{3,1}$ is the locus of pointed curves $[C,p]$ in $\BM_{3,1}$ such that $\mathcal{O}(2p)$ is an odd theta characteristic on $C$.

\begin{thm}
\label{F}
The class of the closure of the locus ${\mathcal{F}}_{3,1}$ in $A^2(\BM_{3,1})$ is
\begin{eqnarray*}
\overline{\mathcal{F}}_{3,1} &\equiv &{}-3\psi^2+ 77\psi\lambda -8\psi\delta_0 -42\psi\delta_{1,1} -19\psi\delta_{2,1} -338\lambda^2 +\frac{137}{2} \lambda\delta_0\\
&& {} +146 \lambda\delta_{2,1} -\frac{7}{2}\delta_0^2  -\frac{31}{2}\delta_0\delta_{2,1} -3\delta_{1,1}^2  -20\delta_{2,1}^2 +3\kappa_2.
\end{eqnarray*}
\end{thm}

Let $\mathcal{W}_{3,1}$ be the divisor of Weierstrass points in $\MM_{3,1}$
\[
\mathcal{W}_{3,1} = \{[C,p]\in\MM_{3,1} \,\,|\,\, 3p+x\sim K_C, \,\, \mbox{for some}\,\, x\in C \}.
\] 
The class of its closure has been computed in \cite{MR1016424}
\[
\overline{\mathcal{W}}_{3,1} \equiv -\lambda +6 \psi -3\delta_{1,1} -\delta_{2,1} \in \Pic(\BM_{3,1}).
\] 
Let $\Theta_{3,1}$ be the divisor of non-hyperelliptic genus-$3$ curves with a marked point lying on one of the $28$ bitangents to the canonical model of the curve. Equivalently, 
\[
\Theta_{3,1} :=\{[C,p]\in\MM_{3,1} \,\,|\,\, p\in {\rm supp}(\eta^-),\,\,\mbox{for some $\eta^-$ an odd theta characteristic} \}.
\]
From \cite[Theorem 0.3]{MR2779475} we have
\[
\overline{\Theta}_{3,1} \equiv 7 \lambda+14\psi -\delta_0 -9\delta_{1,1} -5\delta_{2,1} \in \Pic(\BM_{3,1}).
\]

In order to prove Theorem \ref{F}, we show that $\overline{\mathcal{F}}_{3,1}$ is one of the components of the intersection of $\overline{\mathcal{W}}_{3,1}$  and $\overline{\Theta}_{3,1}$. The statement follows after considering all the components with the respective multiplicity in this intersection.

\subsection{The intersection of \texorpdfstring{$\overline{\mathcal{W}}_{3,1}$}{{W}{3,1}}  and \texorpdfstring{$\overline{\Theta}_{3,1}$}{Theta{3,1}}} 
\label{WTh}
Let us first consider the intersection of ${\mathcal{W}}_{3,1}$  and ${\Theta}_{3,1}$ in the locus of smooth curves $\MM_{3,1}$.

\begin{lem}
The intersection of $\mathcal{W}_{3,1}$  and $\Theta_{3,1}$ consists of two components, corresponding to ${\mathcal{H}yp}_{3,1}$ and ${\mathcal{F}}_{3,1}$.
\end{lem}

\begin{proof}
Let $[C,p]$ be a pointed smooth curve in the intersection of ${\mathcal{W}}_{3,1}$ with ${\Theta}_{3,1}$. Then, there exist $x,y\in C$ such that $3p+x\sim K_C \sim 2p+2y$. If $C$ is hyperelliptic, then we have $x=p$ and $y$ are Weierstrass points. If $C$ is not hyperelliptic, then from $p+x\sim 2y$ we deduce $p=x=y$ and $4p\sim K_C$, hence $p$ is a hyperflex point, and $[C,p]$ is in ${\mathcal{F}}_{3,1}$.
\end{proof}

We now consider the components of the intersection of $\overline{\mathcal{W}}_{3,1}$ and the boundary $\Delta=\BM_{3,1}\setminus \MM_{3,1}$.
Note that the divisor ${\mathcal{W}}_{3,1}$ corresponds to the locus of curves admitting a $\mathfrak{g}^1_3$ with a total ramification at the marked point, and its closure can be studied via admissible covers. 

Let $(C,p,x,y)$ be a smooth $3$-pointed curve of genus $2$. If $[C/_{x\sim y}, p]$ admits a triple admissible cover totally ramified at $p$, then there exists $z\in C$ such that $3p\sim x+y+z$. We denote the closure of the locus of such curves by $(\overline{\mathcal{W}}_{3,1})_0$.

Consider a general element inside $\Gamma_1$: a pointed elliptic curve $(E_1,p)$ meeting an elliptic curve $E_2$ at two points. Such a curve admits a triple admissible cover totally ramified at $p$, whose restriction to $E_2$ has degree $2$. It follows that $\Gamma_1$ is in the intersection 
$\overline{\mathcal{W}}_{3,1}\cap\Delta_0$.

Similarly consider a general element inside $\Gamma_2$: a curve of genus $2$ meeting a rational curve with the marked point in two general points. Such a curve admits a triple admissible cover totally ramified at the marked point, with a simple ramification at one of the two nodes, and no ramification at the other node.

It is easy to see that no other codimension-two boundary locus inside $\Delta_0$ is entirely contained inside $\overline{\mathcal{W}}_{3,1}$, hence we conclude that
\[
\overline{\mathcal{W}}_{3,1}\cap\Delta_0 = (\overline{\mathcal{W}}_{3,1})_0 \cup \Gamma_1 \cup \Gamma_2.
\]

Take an elliptic curve $(E,p,q)$ and attach at $q$ a general genus-$2$ curve $C$. Suppose that such a curve admits a triple admissible cover totally ramified at $p$. There are two cases: if the admissible cover has a simple ramification at $q$, then $q$ is a Weierstrass point in $C$. We denote the closure of the locus of such curves by $(\overline{\mathcal{W}}_{3,1})_{1a}$. If the admissible cover is totally ramified also at $q$, then we have $3p\sim 3q$, that is, $p-q$ is a non-trivial torsion point of order $3$ in ${\rm Pic}^0(E)$. We denote the closure of the locus of such curves by $(\overline{\mathcal{W}}_{3,1})_{1b}$. No other codimension-two boundary locus inside $\Delta_{1,1}$ is contained in $\overline{\mathcal{W}}_{3,1}$, hence we have
\[
\overline{\mathcal{W}}_{3,1}\cap\Delta_{1,1} = (\overline{\mathcal{W}}_{3,1})_{1a} \cup (\overline{\mathcal{W}}_{3,1})_{1b}.
\]

Finally, consider a smooth $2$-pointed genus-$2$ curve $(C,p,q)$ and attach at $q$ an elliptic tail. If such a curve admits a triple admissible cover totally ramified at $p$, then there exists $z\in C\setminus \{p\}$ such that $3p\sim 2q+z$. We denote the closure of this locus by $(\overline{\mathcal{W}}_{3,1})_2$, and we have
\[
\overline{\mathcal{W}}_{3,1}\cap\Delta_{2,1} = (\overline{\mathcal{W}}_{3,1})_2.
\]

Let us now consider the components of the intersection of $\overline{\Theta}_{3,1}$ and the boundary $\Delta$.
Let $(C,p,x,y)$ be a smooth $3$-pointed genus-$2$ curve, and suppose that $[C/_{x\sim y},p]$ is inside $\overline{\Theta}_{3,1}$. There are two possibilities, corresponding to the two components of the space $\mathcal{S}^-_{3}$ over $\Delta_0$. One possibility is that $p$ is in the support of a line bundle $\eta$ such that $\eta^{\otimes 2}=K_C\otimes \mathcal{O}(x+y)$. In this case there exists $z\in C$ such that $2p+2z\sim K_C\otimes \mathcal{O}(x+y)$. We denote the closure of the locus of such curves by $(\overline{\Theta}_{3,1})_{0a}$. The other case is when $p$ is in the support of $\eta$ such that $\eta^{\otimes 2}=K_C$. In this case, $p$ is a Weierstrass point in $C$. We denote the closure of the locus of such curves by $(\overline{\Theta}_{3,1})_{0b}$.

Let $E_1, E_2$ be two elliptic curves meeting in two points, with the marked point in $E_1$. Consider the curve $\overline{E}_2$ in $\BM_{3,1}$ obtained by moving one of the nodes on the curve $E_2$. The intersection of the generating divisor classes is
\begin{align*}
\delta_0 |_{\overline{E}_2} &= -2, & \delta_{2,1} |_{\overline{E}_2} &= 1.
\end{align*}
We conclude that such a curve has negative intersection with $\overline{\Theta}_{3,1}$. Since a deformation of this curve covers an open subset of the locus $\Gamma_1$, we deduce that $\Gamma_1$ is inside $\overline{\Theta}_{3,1}$.

A similar argument holds for the locus $\Gamma_2$. Alternatively, the canonical model of a general curve
in $\Gamma_2$ is an irreducible nodal quartic $C$, and the marked point coincides with the node. %$x=y$.
%We can study the limits of the $28$ bitangents when smooth plane quartics specialize to $C$. If $2p+2q \sim x+y+K_{\widetilde{C}}$, where $\widetilde{C}$ is the normalization of $C$, then there are $2^2 \cdot 2^2 = 16$ solutions for $\{p, q\} \neq \{x, y\}$. The other possibility is that the limit bitangent goes through $x,y$. Then we have $2q \sim K_{\widetilde{C}}$, that is, $q$ is a Weierstrass point of ${\widetilde{C}}$, providing another $6$ limits, each necessarily having multiplicity $2$ in order to get $28$ in total. 
It is classically known that when smooth plane quartics specialize to $C$, the $28$ bitangents specialize to $16$ bitangents meeting $C$ away from the node, and $6$ bitangents of multiplicity two passing through the node 
(see for instance \cite[\S 3]{MR1949642}).
In particular, bitangents can specialize to lines passing through the node,
hence blowing up the node we get a general curve in $\Gamma_2$, which implies that $\Gamma_2$ is contained in $\overline{\Theta}_{3,1}$.

It is easy to see that no other codimension-two boundary locus inside $\Delta_0$ is contained in $\overline{\Theta}_{3,1}$, hence we have
\[
\overline{\Theta}_{3,1}\cap\Delta_0 = (\overline{\Theta}_{3,1})_{0a} \cup (\overline{\Theta}_{3,1})_{0b} \cup \Gamma_1 \cup \Gamma_2.
\]

Let $(E,p,q)$ be a $2$-pointed elliptic curve, and attach at $q$ a general curve of genus $2$. Suppose that such a curve is inside $\overline{\Theta}_{3,1}$. Then such a curve admits $((\mathcal{L}_1,\sigma_1),(\mathcal{L}_2,\sigma_2))$ a limit $\mathfrak{g}^0_{2}$, where $\mathcal{L}_1=\eta_1\otimes\mathcal{O}(2q)$ and $\mathcal{L}_2=\eta_2\otimes\mathcal{O}(q)$, $\eta_1,\eta_2$ are theta characteristics on $E,C$ with opposite parity, $\sigma_i\in H^0(\mathcal{L}_i)$, and $p\in {\rm supp}(\sigma_1)$. Note that if $\eta_2=\eta_2^+$ is even, then ${\rm ord}_q (\sigma_2)=0$, hence by compatibility ${\rm ord}_q (\sigma_1)=2$, which contradicts the fact that $\sigma_1$ vanishes also at $p$. Hence we have that $\eta_2=\eta_2^-$ is odd, and $\eta_1=\eta_1^+$ is even. Since $\eta_1^+$ has no sections, we have ${\rm ord}_q(\sigma_2)\geq 1$. 
There are two cases. If ${\rm ord}_q(\sigma_2)= 2$, then $q$ is a Weierstrass point in $C$. For each $p,q\in E$, we can always find $\eta_1^+$ and $z\in E$ such that $\eta_1^+\otimes \mathcal{O}(2q)=\mathcal{O}(p+z)$. We denote the closure of the locus of such curves by $(\overline{\Theta}_{3,1})_{1a}$.
The other case is ${\rm ord}_q(\sigma_2)= 1$. By compatibility we have ${\rm ord}_q(\sigma_1)= 1$, hence $\eta_1^+\otimes \mathcal{O}(2q)=\mathcal{O}(p+q)$. This implies $\eta_1^+=\mathcal{O}(p-q)$. In particular, $2p\sim 2q$. We denote the closure of the locus of such curves by $(\overline{\Theta}_{3,1})_{1b}$. No other codimension-two boundary locus in $\Delta_{1,1}$ is inside $\overline{\Theta}_{3,1}$, hence we have
\[
\overline{\Theta}_{3,1}\cap\Delta_{1,1} = (\overline{\Theta}_{3,1})_{1a} \cup (\overline{\Theta}_{3,1})_{1b}.
\]

Finally, let $(C,p,q)$ be a smooth $2$-pointed curve of genus $2$, and attach at $q$ an elliptic tail $E$. If such a curve is in $\overline{\Theta}_{3,1}$, then it admits $((\mathcal{L}_1,\sigma_1),(\mathcal{L}_2,\sigma_2))$ a limit $\mathfrak{g}^0_{2}$, where $\mathcal{L}_1=\eta_1\otimes\mathcal{O}(2q)$ and $\mathcal{L}_2=\eta_2\otimes\mathcal{O}(q)$, $\eta_1,\eta_2$ are theta characteristics on $E,C$ with opposite parity, $\sigma_i\in H^0(\mathcal{L}_i)$, and $p\in {\rm supp}(\sigma_2)$. There are two cases. If $\eta_2 = \eta_2^-$ is odd and $\eta_1 = \eta_1^+$ is even, then ${\rm ord}_q (\sigma_1)\leq 1$. By compatibility, we have that ${\rm ord}_q (\sigma_2)\geq 1$. Moreover, since $\sigma_2$ vanishes also at $p$, we deduce that ${\rm ord}_q (\sigma_2)= 1$, and $p$ is in the support of $\eta_2^-$, that is, $p$ is a Weierstrass point in $C$. We denote the closure of the locus of such curves by $(\overline{\Theta}_{3,1})_{2a}$. The other case is when $\eta_2 = \eta_2^+$ is even and $\eta_1 = \eta_1^-$ is odd. Since $h^0(\eta_2^+\otimes \mathcal{O}(q-p))\geq 1$, we have that $(p,q)$ belongs to the Scorza curve $T_{\eta_2^+}$ in $C\times C$. We denote the closure of the locus of such curves by $ (\overline{\Theta}_{3,1})_{2b}$. No other codimension-two boundary locus inside $\Delta_{2,1}$ is inside $\overline{\Theta}_{3,1}$, hence we have proven that
\[
\overline{\Theta}_{3,1}\cap\Delta_{2,1} = (\overline{\Theta}_{3,1})_{2a} \cup  (\overline{\Theta}_{3,1})_{2b}.
\]

Note that $(\overline{\mathcal{W}}_{3,1})_{1a}=(\overline{\Theta}_{3,1})_{1a}$. We rename this locus as $\overline{\mathcal{W}}_2:=(\overline{\mathcal{W}}_{3,1})_{1a}=(\overline{\Theta}_{3,1})_{1a}$. 
A general element of each component of $\overline{\Theta}_{3,1}\cap\Delta$ outside $\overline{\mathcal{W}}_2 \cup \Gamma_1 \cup \Gamma_2$ does not admit a triple admissible cover totally ramified at the marked point. Hence, $\overline{\Theta}_{3,1}\cap\Delta$ meets $\overline{\mathcal{W}}_{3,1}$ in codimension higher than two outside $\overline{\mathcal{W}}_2 \cup \Gamma_1 \cup \Gamma_2$.
We have thus proven the following result.

\begin{lem}
\label{lemmamnklj}
We have 
\begin{align}
\label{mnklj}
\left[ \overline{\mathcal{W}}_{3,1} \right] \cdot \left[\overline{\Theta}_{3,1}\right] = m\cdot \left[\overline{\mathcal{H}yp}_{3,1}\right] + n\cdot  \left[\overline{\mathcal{F}}_{3,1} \right]+ k\cdot \left[\overline{\mathcal{W}}_2\right] + l\cdot \gamma_1 +j \cdot \gamma_2 \in A^2(\BM_{3,1})
\end{align}
for some coefficients $m,n,k,l,j$.
%$\overline{\mathcal{W}}_{3,1} \cap \overline{\Theta}_{3,1} = \overline{\mathcal{H}yp}_{3,1} \cup  \overline{\mathcal{F}}_{3,1} \cup\overline{\mathcal{W}}_2 \cup \Gamma_1 \cup \Gamma_2$ in $\BM_{3,1}$.
\end{lem}

%Hence we can write
Expressing $\lambda\delta_{1,1}$ in terms of the other products of divisor classes (see Proposition \ref{boundaryinA2Mbar31}), the left-hand side of (\ref{mnklj}) is equal to
\begin{eqnarray*}
\left[\overline{\mathcal{W}}_{3,1}\right] \cdot \left[\overline{\Theta}_{3,1}\right] &= & 84 \psi^2 +28 \psi\lambda -6 \psi\delta_0 -\frac{468}{5}\psi\delta_{1,1} -44\psi\delta_{2,1} -7\lambda^2  + \lambda\delta_0\\
 && {} -2\lambda\delta_{2,1}+\frac{9}{5} \delta_0\delta_{1,1} +\delta_0\delta_{2,1} +27\delta_{1,1}^2 +\frac{108}{5} \delta_{1,1}\delta_{2,1} +5\delta_{2,1}^2.
\end{eqnarray*}
The classes  $\gamma_1$ and $\gamma_2$ are in Proposition \ref{boundaryinA2Mbar31}, and the class of $\overline{\mathcal{H}yp}_{3,1}$ comes from Theorem \ref{Hhyp31}.

\begin{lem}
The class of $\overline{\mathcal{W}}_2$ in $A^2(\BM_{3,1})$ is
\[
\left[\overline{\mathcal{W}}_2\right] = {}-\frac{9}{5} \psi\delta_{1,1} -\frac{1}{10} \delta_0\delta_{1,1} -3\delta_{1,1}^2 -\frac{6}{5} \delta_{1,1}\delta_{2,1}.
\]
\end{lem}

\begin{proof}
Let $\xi \colon \BM_{1,2} \times \BM_{2,1}\rightarrow \Delta_{1,1} \subset \BM_{3,1}$ be the map obtained by identifying the marked point on a genus-$2$ curve with the second marked point on an elliptic curve. Let $\pi_1 \colon  \BM_{1,2} \times \BM_{2,1}\rightarrow \BM_{1,2}$ and $\pi_2 \colon  \BM_{1,2} \times \BM_{2,1}\rightarrow \BM_{2,1}$ be the natural projections. Note that $\xi^*(\overline{\mathcal{W}}_2)$ is the pull-back via $\pi_2$ of the Weierstrass divisor in $\BM_{2,1}$, hence
$\xi^*(\overline{\mathcal{W}}_2) \equiv \pi_2^* \left(3  \psi -\frac{1}{10} \delta_0 -\frac{6}{5} \delta_1 \right).$
On the other hand, $A^1(\Delta_{1,1})$ is generated by the classes $\psi\delta_{1,1}, \delta_0\delta_{1,1}, \delta_{1,1}^2, \delta_{1,1}\delta_{2,1}$, and we have the following pull-back formulae
\begin{align*}
\xi^*(\psi\delta_{1,1}) &= \pi_1^*(\psi_1), & \xi^*(\delta_0\delta_{1,1}) &= \pi_1^*(\delta_0) +\pi_2^*(\delta_0), \\
\xi^*(\delta_{1,1}^2) &= -\pi_1^*(\psi_2) - \pi_2^*(\psi), & \xi^*(\delta_{1,1}\delta_{2,1}) &= \pi_1^*(\delta_{0,\{1,2\}}) +\pi_2^*(\delta_{1,1}).
\end{align*}
Moreover, we have $\psi_i = \frac{1}{12} \delta_0 + \delta_{0,\{1,2\}}$  in $A^1(\BM_{1,2})$,
for $i=1,2$. Since the maps $\xi^*\colon A^1(\Delta_{1,1})\rightarrow A^1 ( \BM_{1,2} \times \BM_{2,1})$ and $A^1(\Delta_{1,1})\rightarrow A^2(\BM_{3,1})$ are injective, the statement follows.
\end{proof}

\subsection{Test surfaces} 
\label{TSMbar31}
In this section, restricting (\ref{mnklj}) to some test surfaces, we deduce linear relations among the coefficients $m,n,k,l,j$.

\subsubsection{} 
\label{TS5Mbar31}
Let $(E,q)$ be a pointed elliptic curve. Identify the point $q$ with a moving point $y$ on a general curve $C$ of genus $2$, and consider a moving marked point $x$ in $E$.

\begin{figure}[htbp]
\centering
\psfrag{$C$}[b][b]{$C$}
\psfrag{$E$}[b][b]{$E$}
 \includegraphics[scale=0.7]{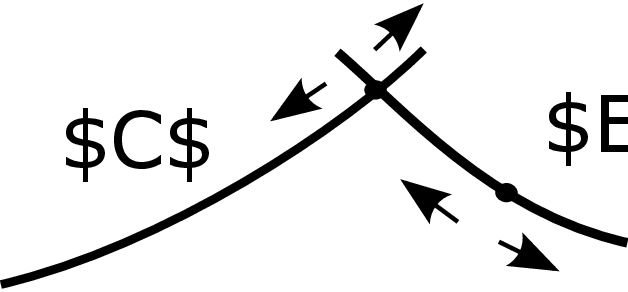}
\end{figure}

The base of this family is $E\times C=T_1$. Let $\pi_1\colon E\times C \rightarrow E$ and $\pi_2\colon E\times C \rightarrow C$ the natural projections.
The non-zero divisors are
\begin{align*}
\psi |_{T_1} &= \pi_1^*([q]), & \delta_{1,1} |_{T_1} &= -\pi_1^*([q])-\pi_2^*(K_C), & \delta_{2,1}  |_{T_1} &= \pi_1^*([q]).
\end{align*}
We deduce
\begin{align*}
\psi\delta_{1,1}  |_{T_1} &= -2, & \delta_{1,1}^2  |_{T_1} &= 4, & \delta_{1,1} \delta_{2,1}  |_{T_1} &= -2,
\end{align*}
and
\begin{align*}
\left[\overline{\mathcal{W}}_{3,1}\right] \cdot \left[\overline{\Theta}_{3,1}\right]  |_{T_1} &= 252, & \left[\overline{\mathcal{H}yp}_{3,1}\right]  |_{T_1} &= 18, & \left[\overline{\mathcal{W}}_2 \right]  |_{T_1} &= -6, & \gamma_1 |_{T_1} =\gamma_2  |_{T_1} &=0.
\end{align*}

Let us consider the restriction of the class of $\overline{\mathcal{F}}_{3,1}$ to this test surface. If a fiber $(E\cup_y C,x)$ of this family is in the intersection with $\overline{\mathcal{F}}_{3,1}$, then it admits $((\mathcal{L}_E, \sigma_E),(\mathcal{L}_C, \sigma_C))$ a limit $\mathfrak{g}^0_{2}$ such that $\mathcal{L}_E=\eta_E\otimes\mathcal{O}(2y)$, $\mathcal{L}_C=\eta_C\otimes\mathcal{O}(y)$, $\eta_E$ and $\eta_C$ are theta characteristics on $E$,  $C$ with opposite parity, and $\sigma_E\in H^0(\mathcal{L}_E)$, $\sigma_C\in H^0(\mathcal{L}_C)$. Moreover, we have ${\rm ord}_x \sigma_E =2$. This implies ${\rm ord}_y \sigma_E =0$, hence ${\rm ord}_y \sigma_C =2$. It follows that $\eta_C=\eta_C^-$ is an odd theta characteristic, $y$ is a Weierstrass point on $C$, and $\eta_E=\eta_E^+$ is even. Since $\eta_E^+\otimes \mathcal{O}(2y) = \mathcal{O}(2x)$, we deduce $\eta_E^+=\mathcal{O}(2x-2y)$. The map $E\times E \rightarrow {\rm Pic}^0(E)$ defined as $(x,y)\mapsto \mathcal{O}(2x-2y)$ has degree $4$. We conclude that there are $6\cdot 4 \cdot 3$ fibers of this family in the intersection with $\overline{\mathcal{F}}_{3,1}$. Since this family lies in the locus of compact type, each fiber counts with multiplicity $1$ (see for instance \cite[Lemma 3.4 and the following Remark]{MR985853}).
We deduce the following relation
\[
252 = 18m + 72n -6k.
\]

\subsubsection{}
\label{TS9Mbar31} 
Let $C$ be a general curve of genus $2$. Let $x,y\in C$ be two moving points in $C$. Consider the surface obtained by attaching an elliptic tail at the point $y$.

\begin{figure}[htbp]
\centering
\psfrag{$C$}[b][b]{$C$}
 \includegraphics[scale=0.7]{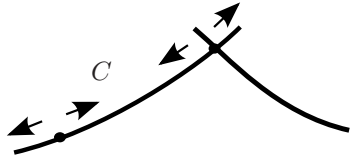}
  \caption{How the general fiber of the family in \S \ref{TS9Mbar31} moves.}
\end{figure}

The base of this family is $C\times C=:T_2$. Let $\pi_i\colon C\times C \rightarrow C$ be the natural projection, for $i=1,2$. The non-zero restrictions of the divisor classes are
\begin{align*}
\psi  |_{T_2} &= \pi_1^*(K_C)+\Delta_{C\times C}, & \delta_{1,1} |_{T_2} &= \Delta_{C\times C}, & \delta_{2,1} |_{T_2} &= -\pi_1^*(K_C)-\Delta_{C\times C}. \\
\end{align*}
Hence we deduce
\begin{align*}
\psi^2 |_{T_2} &= 2, & \delta_{2,1}^2 |_{T_2} &= 2, & \psi\delta_{2,1} |_{T_2} &= -6, & \delta_{1,1}^2 |_{T_2} &= -2.
\end{align*}
Note that the restriction of $\kappa_2$ is equal to the restriction of $\kappa_2$ to the surface in $\BM_{2,2}$ obtained by forgetting the elliptic tail. After (\ref{kappa2Mbar22}), the class $\kappa_2$ is equivalent to a product of divisor classes in $\BM_{2,2}$, hence we easily compute
\[
\kappa_2 |_{T_2} =2.
\]
We deduce
\begin{align*}
\left[\overline{\mathcal{W}}_{3,1}\right] \cdot \left[\overline{\Theta}_{3,1}\right] |_{T_2} &= 388, & \left[\overline{\mathcal{H}yp}_{3,1}\right] |_{T_2} &= 30, & \left[\overline{\mathcal{W}}_2 \right] |_{T_2} &= 6, & \gamma_1|_{T_2} =\gamma_2 |_{T_2} &=0.
\end{align*}

If a fiber $(E\cup_y C,x)$ of this family is in the intersection with $\overline{\mathcal{F}}_{3,1}$, then it admits a limit linear series $((\mathcal{L}_E, \sigma_E),  (\mathcal{L}_C, \sigma_C))$ of type $\mathfrak{g}^0_{2}$ such that $\mathcal{L}_E=\eta_E\otimes\mathcal{O}(2y)$, $\mathcal{L}_C=\eta_C\otimes\mathcal{O}(y)$, $\eta_E$ and $\eta_C$ are theta characteristics on $E$,  $C$ with opposite parity, $\sigma_E\in H^0(\mathcal{L}_E)$, $\sigma_C\in H^0(\mathcal{L}_C)$, and ${\rm ord}_x \sigma_C =2$. This implies ${\rm ord}_y \sigma_C =0$ and ${\rm ord}_y \sigma_E =2$, hence $\eta_E=\eta_E^-=\mathcal{O}_E$ is odd and $\eta_C=\eta_C^+$ is even. Note that $\eta_C^+\otimes\mathcal{O}(y)=\mathcal{O}(2x)$. Since the difference map $C\times C\rightarrow {\rm Pic}^0(C)$ defined as $(x,y)\mapsto \mathcal{O}(2x-y)$ has degree $8$ (see Proposition \ref{diff}), we deduce that there are $8\cdot 10$ fibers of this family in the intersection with $\overline{\mathcal{F}}_{3,1}$. Each fiber counts with multiplicity one, and we have
\[
388 = 30m +80n +6k.
\]

\subsubsection{} 
\label{TS13Mbar31}
Consider a chain of $3$ elliptic curves, with a marked point on an external component. Vary the central elliptic component in a pencil of degree $12$, and vary one of the singular points on the central component.

\begin{figure}[htbp]
\centering
 \includegraphics[scale=0.7]{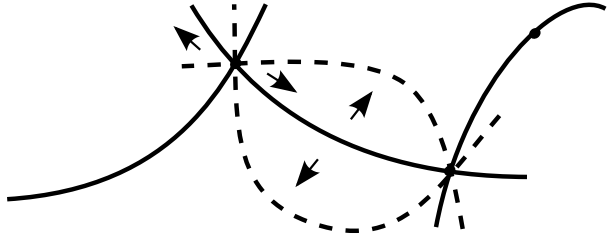}
  \caption{How the general fiber of the family in \S \ref{TS13Mbar31} moves.}
\end{figure}

The base of this surface is the blow-up $T_3$ of $\mathbb{P}^2$ at the $9$ points of intersection of two general cubics. Let $H$ be the pull-back of an hyperplane section in $\mathbb{P}^2$, let $\Sigma$ be the sum of the $9$ exceptional divisors, and $E_0$ one of them. We have
\begin{align*}
\lambda |_{T_3} &= 3H-\Sigma, & \delta_0 |_{T_3} &= 36 H-12\Sigma,&
\delta_{1,1} |_{T_3} &= {}-3H +\Sigma -E_0, & \delta_{2,1} |_{T_3} &= {}-3H+\Sigma
\end{align*}
(see \cite[\S 3 ($\lambda$)]{MR1078265}). We deduce
\begin{align*}
\delta_0\delta_{1,1} |_{T_3} &= -12, & \delta_{1,1}^2 |_{T_3} &= 1, & \delta_{1,1}\delta_{2,1} |_{T_3} &= 1,
\end{align*}
and moreover we have
\begin{align*}
\delta_{01a} |_{T_3} &=12, & \kappa_2 |_{T_3} &= 1.
\end{align*}
Indeed, there are $12$ fibers of this family contributing to the intersection with $\delta_{01a}$, namely when the central elliptic component degenerates to a rational normal curve and the moving node collides with the other non-disconnecting node. This family has zero intersection with $\delta_{00}$, hence we deduce the restriction of $\kappa_2$ from Proposition \ref{boundaryinA2Mbar31}.
Finally, we have
\begin{align*}
\left[\overline{\mathcal{W}}_{3,1}\right] \cdot \left[\overline{\Theta}_{3,1}\right] |_{T_3} &= 27, & \left[\overline{\mathcal{H}yp}_{3,1}\right] |_{T_3} &= 0, & \left[\overline{\mathcal{W}}_2 \right] |_{T_3} &= -3, & \gamma_1 |_{T_3} &=12, &\gamma_2 |_{T_3} &=0.
\end{align*}

Since we can choose the marked point generically in one of the elliptic tails, from an analysis similar to the one in \S \ref{TS5Mbar31}, we deduce that this family is disjoint from $\overline{\mathcal{F}}_{3,1}$.
Hence, we have
\[
27 = -3k +12 l.
\]

\subsection{Push-forward to \texorpdfstring{$\BM_3$}{BM3}}
\label{pushfwdtoM3}

In this section, we compute the push-forward of (\ref{mnklj}) via the forgetful map $p\colon \BM_{3,1}\rightarrow \BM_3$.
We use the following formulae
\begin{align*}
p_*(\psi^2) &= \kappa_1 = 12 \lambda -\delta_0 - \delta_1, & p_*(\psi\lambda) &= 4\lambda, & p_*(\psi\delta_0) &= 4\delta_0,\\
p_*(\kappa_2) &= p_*(p^*(\kappa_2)+\psi^2) = \kappa_1, & p_*(\psi\delta_{1,1})&= \delta_1, & p_*(\psi\delta_{2,1})&= 3\delta_1,\\
p_*(\delta_{1,1}^2) &= -\delta_1, & p_*(\delta_{1,1}\delta_{2,1})&= \delta_1, & p_*(\delta_{2,1}^2) &= -\delta_1.
\end{align*}
All other classes in Proposition \ref{basisR2Mbar31} have zero push-forward (see for instance \cite[Theorem 2.8]{MR1953519}). We deduce the following
\[
p_*(\overline{\mathcal{W}}_{3,1} \cdot \overline{\Theta}_{3,1}) = 1120 \lambda -108 \delta_0 -320\delta_1.
\]
Moreover, we have 
\begin{align*}
p_*(\overline{\mathcal{H}yp}_{3,1}) &= 8(9\lambda-\delta_0-3\delta_1), & p_*(\Gamma_2) &= \delta_0, & p_*(\Gamma_1)=p_*(\overline{\mathcal{W}}_2) &=  0.
\end{align*} 

The push-forward of the locus $\overline{\mathcal{F}}_{3,1}$ coincides with the push-forward via the forgetful map $\pi\colon \mathcal{S}^-_3\rightarrow \BM_3$ of the closure of the divisor $\mathcal{Z}_3$ in $\mathcal{S}^-_3$ of curves with an odd spin structure vanishing twice at a certain point. This class has been computed in \cite[pg.~345]{MR1016424}:
\[
p_*(\overline{\mathcal{F}}_{3,1}) = \pi_*(\overline{\mathcal{Z}}_3) \equiv 308 \lambda -32\delta_0 -76\delta_1 \in \Pic(\BM_{3})
\]
(see also \cite[Theorem 0.5]{MR3245010}).
From equation (\ref{mnklj}), we conclude
\begin{align*}
m &= 7, & n &= 2, & j &= 12.
\end{align*}

\subsection{Proof of Theorem \ref{F}} 
\label{proofF}
In \S \ref{WTh}, we have realized the locus $\overline{\mathcal{F}}_{3,1}$ as one of the components of the intersection of the divisors $\overline{\mathcal{W}}_{3,1}$ and $\overline{\Theta}_{3,1}$. In order to compute the class of $\overline{\mathcal{F}}_{3,1}$ it remains to find the multiplicities $m,n,k,l,j$ in (\ref{mnklj}).
From the study of the push-forward to $\BM_3$ in \S \ref{pushfwdtoM3}, we have found that $m=7$, $n=2$ and $j=12$. In \S \ref{TSMbar31}, using test surfaces we have three linear relations involving $m,n,k,l$. We deduce $k=l=3$, and we have one more relation as a check. The statement follows. \hfill$\square$

\section{The locus \texorpdfstring{$\overline{\mathcal{H}}^+_4$}{BH+4}}
\label{h+4}

In this section, we compute the class of the closure in $\BM_4$ of the locus
\[
\mathcal{H}^+_4 := \{[C]\in \MM_4 \,|\,  \mathcal{O}(3x) \,\, \mbox{is an even theta characteristic for some $x\in C$} \}.
\]
The strategy is similar to the one used to compute the class of $\overline{\mathcal{F}}_{3,1}$ in \S \ref{SF}, that is, we realize $\overline{\mathcal{H}}^+_4$ as a component of the intersection of two divisors. Namely, we consider the divisor of curves with a vanishing theta-null
\[
\Theta_{\rm null} :=\{ [C] \in \MM_4 \, | \, h^0(\eta^+)>0   \,\, \mbox{for some $\eta^+$ an even theta characteristic} \}
\]
and the divisor of curves admitting a $\mathfrak{g}^1_3$ with a total ramification point
\[
\mathcal{T} := \{ [C]\in \MM_4 \, | \, h^0(\mathcal{O}(3x))\geq 2 \,\, \mbox{for some $x\in C$}  \}.
\]
The classes of the closures of $\Theta_{\rm null}$ and $\mathcal{T}$ in $\BM_4$ are special cases of divisor classes computed in \cite{MR937985} and \cite{MR791679}, respectively:
\begin{align*}
\overline{\Theta}_{\rm null}  & \equiv  34 \lambda - 4 \delta_0 - 14 \delta_1 - 18 \delta_2, &
\overline{\mathcal{T}} & \equiv  264 \lambda - 30 \delta_0 - 96 \delta_1 - 128 \delta_2. 
\end{align*}

\subsection{The intersection of \texorpdfstring{$\overline{\Theta}_{\rm null}$}{Theta{null}} and \texorpdfstring{$\overline{\mathcal{T}}$}{T}}
\label{intThT}
Let us first consider the intersection of ${\Theta}_{\rm null}$ and ${\mathcal{T}}$ in $\MM_4$.

\begin{lem}
The intersection of $\Theta_{\rm null}$ and $\mathcal{T}$ in $\MM_4$ has two components, corresponding to $\mathcal{H}yp_4$ and $\mathcal{H}_4^+$. 
\end{lem}

\begin{proof}
It is clear that $\mathcal{H}yp_4$ and $\mathcal{H}_4^+$ are contained in both $\Theta_{\rm null}$ and $\mathcal{T}$. Conversely, suppose that $C$ is a smooth curve contained in both $\Theta_{\rm null}$ and $\mathcal{T}$. In particular, there exists $x$ in $C$ such that $h^0(\mathcal{O}(3x))\geq 2$. If $C$ is not hyperelliptic, $3x$ admits a unique $\mathfrak{g}^1_3$ equal to its canonical residual, i.e.~$6x \sim K_C$, hence $C$ is in $\mathcal{H}_4^+$. 
\end{proof}

Next, we analyze the intersection of the divisor $\overline{\Theta}_{\rm null}$ and the boundary $\Delta:=\BM_4\setminus \MM_4$. Note that the divisor $\overline{\Theta}_{\rm null}$ of genus-$4$ curves with a vanishing theta-null coincides with the Gieseker-Petri divisor of genus-$4$ curves whose canonical model lies on a quadric cone.

Let $(C,x,y)$ be a two-pointed curve of genus $3$, and suppose that $[C/_{x \sim y}]$ is inside $\overline{\Theta}_{\rm null}$. One possibility is that there exists $r$ in $C$ such that $2(p+q+r)\sim K_C\otimes \mathcal{O}(p+q)$. We denote the locus of such curves by $(\overline{\Theta}_{\rm null})_{0a}$. Equivalently, this is the locus of irreducible nodal curves whose canonical model lies on a quadric cone, with the node being away from the vertex of the cone. The other possibility is that $2(p+q)\sim K_C$, that is, $C$ is hyperelliptic. We denote the locus of such curves by $(\overline{\Theta}_{\rm null})_{0b}$. This corresponds to the locus of irreducible nodal curves whose canonical model lies in a quadric cone and the node coincides with the vertex.

Let $\Gamma_1$ be the closure of the locus of curves with an elliptic component meeting a component of genus $2$ in two points. Consider the one-dimensional family $E$ of curves obtained by identifying two general points on a general curve of genus $2$ with a fixed point and a moving point on an elliptic curve. The restriction of the generating divisor classes is
\begin{align*}
\delta_0 |_{E} &= -2, & \delta_1 |_{E} &= 1.
\end{align*}
It follows that such a family has negative intersection with $\overline{\Theta}_{\rm null}$. Since this family produces a moving curve in $\Gamma_1$, we deduce that $\Gamma_1$ is contained inside $\overline{\Theta}_{\rm null}$.

No other codimension-two boundary component inside $\Delta_0$ is entirely contained in $\overline{\Theta}_{\rm null}$, hence we have
\[
\overline{\Theta}_{\rm null} \cap \Delta_0 = (\overline{\Theta}_{\rm null})_{0a} \cup (\overline{\Theta}_{\rm null})_{0b}\cup \Gamma_1.
\]

Next, let $(C,y)$ be a smooth pointed genus-$3$ curve, and let $(E,y)$ be an elliptic curve. If $[C\cup_y E]$ is in $\overline{\Theta}_{\rm null}$, then it admits a limit linear series $(l_C, l_E) = ((\mathcal{L_C}, V_C), (\mathcal{L}_E, V_E))$ of type $\mathfrak{g}^1_3$, where $\mathcal{L}_C=\eta_C \otimes \mathcal{O}(y)$, $\mathcal{L}_E=\eta_E \otimes \mathcal{O}(3y)$, and $\eta_C, \eta_E$ are theta characteristics on $C,E$, respectively, with same parity. Note that one has $a^{l_C}(y)\geq (0,2)$. If $a_0^{l_C}(y)= 0$ and $a_1^{l_C}(y) \geq 2$, we deduce that $y$ is in the support of $\eta_C$, and $\eta_C=\eta_C^-$ (as well as $\eta_E = \eta_E^-$) is an odd theta characteristic. We denote the locus of such curves by $(\overline{\Theta}_{\rm null})_{1a}$. This corresponds to the locus of cuspidal curves in a quadric cone such that the cusp is not the vertex. If $a^{l_C}(y)= (1,3)$, then $C$ is a  hyperelliptic curve, and $y$ is a Weierstrass point in $C$. We denote the locus of such curves by $(\overline{\Theta}_{\rm null})_{1b}$. This corresponds to the locus of cuspidal curves in a quadric cone such that the cusp is the vertex. No other codimension-two boundary component inside $\Delta_1$ is entirely contained in $\overline{\Theta}_{\rm null}$, hence we have
\[
\overline{\Theta}_{\rm null} \cap \Delta_1 = (\overline{\Theta}_{\rm null})_{1a}\cup (\overline{\Theta}_{\rm null})_{1b}.
\]

Suppose $[C_1\cup_y C_2]$ is in $\overline{\Theta}_{\rm null}$, where $C_1$ and $C_2$ are two smooth curves of genus $2$ attached at a point $y$. Then $[C_1\cup_y C_2]$ admits a limit linear series $(l_{C_1}, l_{C_2}) = ((\mathcal{L_{C_1}}, V_{C_1}), (\mathcal{L}_{C_2}, V_{C_2}))$ of type $\mathfrak{g}^1_3$, where $\mathcal{L}_{C_i}=\eta_{C_i} \otimes \mathcal{O}(2y)$ for $i=1,2$, and $\eta_{C_1}, \eta_{C_2}$ are theta characteristics on $C_1, C_2$, respectively, with same parity. If $a^{l_{C_1}}(y)= (0,2)$, then $a^{l_{C_2}}(y)= (1,3)$, and the only other possibility is obtained by switching the two curves. This implies that $y$ is a Weierstrass point on $C_2$. We denote the locus of such curves by $(\overline{\Theta}_{\rm null})_2$. We have
\[ 
\overline{\Theta}_{\rm null} \cap \Delta_2 = (\overline{\Theta}_{\rm null})_2. 
\]

Finally, we consider the intersection of the divisor $\overline{\mathcal{T}}$ and the boundary $\Delta$. The closure of ${\mathcal{T}}$ corresponds to curves admitting a triple admissible cover  totally ramified at some nonsingular point $x$.

If $[C/_{x\sim y}]$ is an irreducible nodal curve in $\overline{\mathcal{T}}$, then there exist $x$ and $z$ in $C$ such that $x+y+z\sim 3x$. We denote the locus of such curves by $\overline{\mathcal{T}}_{0}$. Moreover, a general curve inside $\Gamma_1$ has a triple admissible cover totally ramified at a point $x$ in the elliptic component, with a simple ramification at one of the two nodes, and no ramification at the other node. This is the only codimension-two boundary component inside $\Delta_0$ and $\overline{\mathcal{T}}$, hence we have
\[
\overline{\mathcal{T}}  \cap \Delta_0 = \overline{\mathcal{T}}_{0}\cup \Gamma_1.
\]

Let $(C,y)$ be a pointed curve of genus $3$, and let $(E,y)$ be an elliptic curve. Suppose that $[C\cup_y E]$ is in $\overline{\mathcal{T}}$, that is, $[C\cup_y E]$ has a triple admissible cover totally ramified at some point $x$. There are two cases. If $x$ is in $C$, then there exists $r$ in $C$ such that 
$3x\sim 2y+r$. We denote the locus of such curves by $\overline{\mathcal{T}}_{1a}$. If $x$ is in $E$, then $y$ is a Weierstrass point in $C$. Note that $x-y$ is a nontrivial $3$-torsion point in ${\rm Pic}^0(E)$. We denote the locus of such curves by $\overline{\mathcal{T}}_{1b}$. There are no other codimension-two boundary components inside $\Delta_1$ and $\overline{\mathcal{T}}$, hence we have
\[
\overline{\mathcal{T}} \cap \Delta_1 = \overline{\mathcal{T}}_{1a}\cup \overline{\mathcal{T}}_{1b}.
\]

Consider the stable curve $[C_1\cup_y C_2]$ obtained by identifying a point on two smooth curves $C_1$ and $C_2$ of genus $2$. Suppose $[C_1\cup_y C_2]$ admits a triple admissible cover totally ramified at some point $x$ in $C_1$. If the restriction of the cover to $C_2$ has degree $2$, then $y$ is a Weierstrass point in $C_2$. We denote the locus of such curves by $\overline{\mathcal{T}}_{2a}$. Otherwise, if the restriction of the cover to $C_2$ has degree $3$, then $3x\sim 3y$ on $C_1$. This is a codimension-one condition on $(C_1,y)$. We denote the locus of such curves by  $\overline{\mathcal{T}}_{2b}$, and we have
\[
\overline{\mathcal{T}} \cap \Delta_2 = \overline{\mathcal{T}}_{2a}\cup \overline{\mathcal{T}}_{2b}. 
\]

Note that $\overline{\mathcal{T}}_{2a} = (\overline{\Theta}_{\rm null})_2$. Let us define $\overline{\mathcal{W}}_2:= \overline{\mathcal{T}}_{2a} = (\overline{\Theta}_{\rm null})_2$. 
A general element of each component of $\overline{\Theta}_{\rm null}\cap\Delta$ outside $\overline{\mathcal{W}}_2 \cup \Gamma_1$ does not admit a triple admissible cover totally ramified at  some smooth point. Hence, $\overline{\Theta}_{\rm null}\cap\Delta$ meets $\overline{\mathcal{T}}$ in codimension higher than two outside $\overline{\mathcal{W}}_2 \cup \Gamma_1$.
We thus obtain the following result.

\begin{lem}
We have 
\begin{align}
\label{mnkl}
\left[\overline{\Theta}_{\rm null}\right] \cdot \left[\overline{\mathcal{T}}\right] = m \cdot \left[\overline{\mathcal{H}yp}_4\right] + n \cdot \left[\overline{\mathcal{H}}_4^{+} \right] + k \cdot \left[\overline{\mathcal{W}}_2 \right]+ l\cdot \gamma_1 \in A^2 (\BM_4)
\end{align}
for some coefficients $m,n,k,l$.
%$\overline{\Theta}_{\rm null} \cap \overline{\mathcal{T}} = \overline{\mathcal{H}yp}_4 \cup \overline{\mathcal{H}}_4^{+}\cup \overline{\mathcal{W}}_2 \cup \Gamma_1$. 
\end{lem}

%Following \cite{MR1078265}, we use the basis of $A^2(\BM_4)$ given by the classes in (\ref{basisA2M4}).
Using the relation $(10\lambda - \delta_0 - 2\delta_1 ) \delta_2 = 0$ in $A^2(\BM_4)$, we can write the left-hand side of (\ref{mnkl}) as
\begin{eqnarray*}
\left[\overline{\Theta}_{\rm null}\right] \cdot \left[\overline{\mathcal{T}} \right] & = & 8976\lambda^2 - 2076 \lambda\delta_0 - 6960\lambda\delta_1  + 1416\lambda\delta_2 +
120 \delta_0^2 \\
&&{}+ 804\delta_0\delta_1  + 1344\delta_1^2   + 1416 \delta_1\delta_2 + 2304 \delta_2^2.  
\end{eqnarray*}
The class of $\overline{\mathcal{H}yp}_4$ is in (\ref{hyp4}). It remains to compute the class of $\overline{\mathcal{W}}_2$.

\begin{lem}
The class of $\overline{\mathcal{W}}_2$ in $A^2(\BM_4)$ is  
\[
\left[\overline{\mathcal{W}}_2\right] =  {}- \lambda\delta_2 - \delta_1 \delta_2 - 3\delta_2^2 . 
\]
\end{lem}

\begin{proof}
Let $\xi \colon \BM_{2,1}\times \BM_{2,1} \to \Delta_2\subset \BM_4$ be the gluing morphism, and let $\pi_i\colon \BM_{2,1}\times \BM_{2,1} \to \BM_{2,1}$ be the natural projection on the $i$-th factor, for $i=1,2$. Note that $\xi^{-1}(\overline{\mathcal{W}}_2)$ is the union of the pull-backs of the Weierstrass divisor from both factors. Recall that the Weierstrass divisor in $\BM_{2,1}$ has class $3\psi - \frac{1}{10}\delta_0 -\frac{6}{5}\delta_1$. This implies that 
\[
\xi^* \left( \left[\overline{\mathcal{W}}_2 \right]\right) = 3 \left( \pi_1^*(\psi) + \pi_2^*(\psi) \right) - \frac{1}{10}\left(\pi_1^*(\delta_0) + \pi_2^*(\delta_{0})\right) - \frac{6}{5} \left(\pi_1^*(\delta_{1}) + \pi_2^*(\delta_{1})\right). 
\]
Note that $A^1(\Delta_2)$ is generated by 
$\delta_0\delta_2$,  $\delta_1\delta_2$,  and $\delta_2^2$. 
Moreover, 
\begin{align*}
 \xi^{*}(\delta_{0}\delta_2) &= \pi_1^* (\delta_0) + \pi_2^*(\delta_{0}), & \xi^{*} (\delta_{1}\delta_2) &= \pi_1^*(\delta_1) + \pi_2^*(\delta_{1}), & \xi^{*} (\delta_2^2) &= {}- \pi_1^*(\psi) - \pi_2^*(\psi). 
\end{align*}
We thus conclude that 
\begin{eqnarray*}
\left[\overline{\mathcal{W}}_2  \right] =    {}- \frac{1}{10} \delta_0\delta_2 - \frac{6}{5}\delta_1\delta_2 - 3 \delta_2^2  =  {}- \lambda\delta_2 - \delta_1 \delta_2 - 3\delta_2^2 
\end{eqnarray*}
in $A^1(\Delta_2)$. Since $A^1(\Delta_2) \to A^2(\BM_4)$ is injective, the same formula holds in $A^2(\BM_4)$. 
\end{proof}

\subsection{Test surfaces} 
\label{TS4}
In this section we restrict (\ref{mnkl}) to four test surfaces in order to compute the coefficients $m,n,k,l$.

\subsubsection{}
\label{H4-1}
Let us consider the test surface in $\overline{\mathcal{M}}_4$ obtained by attaching two general curves $C_1$ and $C_2$ of genus $2$ at one point $y$, and moving the point $y$ on both curves. The base of the family of such curves is $C_1\times C_2=: V_1$. We denote by $\pi_i\colon C_1\times C_2 \rightarrow C_i$ the projection to the $i$-th component, for $i=1,2$.

\begin{figure}[htbp]
\centering
 \includegraphics[scale=0.7]{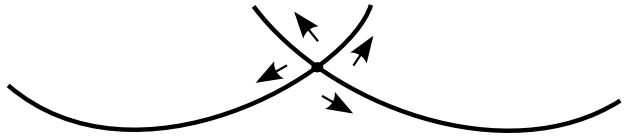}
  \caption{How the general fiber of the family in \S \ref{H4-1} moves.}
\end{figure}

The intersection of this test surface with $\delta_2^2$ is
\[
\delta_2^2 |_{V_1} = (\pi_1^*K_{C_1}\otimes \pi_2^*K_{C_2})^2 =8
\]
while all other generating classes restrict to zero. We deduce that
\begin{align*}
\left[\overline{\Theta}_{\rm null}\right] \cdot \left[\overline{\mathcal{T}}\right] |_{V_1} &= 18432, &   \left[\overline{\mathcal{H}yp}_4\right] |_{V_1} &= 36, & \left[\overline{\mathcal{W}}_2\right] |_{V_1} &= -24, & \gamma_1 |_{V_1} &=0.
\end{align*}

Let us consider the intersection of $\overline{\mathcal{H}}^+_4$ with this test surface.
If an element $C_1 \cup_y C_2$ of this family lies in the intersection of $\overline{\mathcal{H}}^+_4$, then it admits $(l_{C_1}, l_{C_2})=((\mathcal{L}_{C_1}, V_{C_1}),(\mathcal{L}_{C_2}, V_{C_2}))$ a limit $\mathfrak{g}^1_{3}$ with $\mathcal{L}_{C_i}=\eta_i\otimes \mathcal{O}(2y)$, for $i=1,2$, where $\eta_1, \eta_2$ are theta characteristics of the same parity, respectively on $C_1, C_2$. Moreover, one of the two linear series $\mathfrak{g}^1_{3}$ has a section vanishing with order $3$ at some point $x$.

Suppose that $x$ specializes to the singular point $y$ of $C_1 \cup_y C_2$. Consider the pointed curve $(C_1 \cup R \cup C_2, x)$ in $\BM_{4,1}$ lying above $C_1 \cup_y C_2$, where $R$ is a rational component connecting $C_1$ and $C_2$ and containing $x$. If $C_1 \cup_y C_2$ is in $\overline{H}^+_4$, then $(C_1 \cup R \cup C_2, x)$ admits a limit $\mathfrak{g}^3_6$ whose $R$-aspect has a section vanishing with order $6$ at $x$. It is easy to see that this violates the Pl\"ucker formula for the total number of ramification points of a linear series on a rational curve.
Hence, in the following we assume that $x$ is a smooth point.

Suppose that $x$ is in $C_1$ (the case $x$ in $C_2$ is analogous and will multiply the final answer by a factor $2$). There are two cases.

First, consider the case when $\eta_1, \eta_2$ are even theta characteristics $\eta_1^+, \eta_2^+$. Since $h^0(\eta_1^+\otimes \mathcal{O}(2y-3x))= 1$, we have $a_0^{l_{C_1}}(y)=0$, hence $a_1^{l_{C_2}}(y)=3$. This implies $h^0(\eta_2^+\otimes \mathcal{O}(-y))= 1$, a contradiction, since $\eta_2^+$ is an even theta characteristic on a general curve.

Next, consider the case when $\eta_1, \eta_2$ are odd theta characteristics $\eta_1^-, \eta_2^-$. Similarly as before, we have that $h^0(\eta_1^-\otimes \mathcal{O}(2y-3x))\geq 1$, hence $a_0^{l_{C_1}}(y)=0$, and $a_1^{l_{C_2}}(y)=3$. This implies $h^0(\eta_2^-\otimes \mathcal{O}(-y))= 1$, that is, $y$ is in the support of $\eta_2^-$, that is, $y$ is a Weierstrass point of $C_2$. Since $h^0(\eta_2^-\otimes \mathcal{O}(-y))= 1$, from Riemann-Roch we have $h^0(\eta_2^-\otimes \mathcal{O}(y))=h^0(\eta_2^-\otimes \mathcal{O}(2y))= 2$, hence $a_0^{l_{C_2}}(y)=1$. This implies $a_1^{l_{C_1}}(y)= 2$ (in particular $y$ is not in the support of $\eta_1^-$). 
Hence there exists a point $z$ in $C_1$ such that $\mathcal{O}(3x)=\mathcal{O}(2y+z)=\eta_1^-\otimes\mathcal{O}(2y)$.
This implies $\mathcal{O}(3x-2y)=\eta_1^-=\mathcal{O}(z)$. By Proposition \ref{diff}, the map $C_1\times C_1 \rightarrow {\rm Pic}^1(C_1)$ given by $(x,y)\mapsto \mathcal{O}(3x-2y)$ is surjective of degree $72$. 

Moreover, since $y$ is not in the support of $\eta_1^-$, one has to exclude the pairs $(x,y)$ such that $3x\sim 3y$ and $y$ is in the support of $\eta^-$, that is, $y$ is a Weierstrass point. These conditions imply that $x=y$ is a Weierstrass point. By Proposition \ref{diff}, the map $C_1\times C_1 \rightarrow {\rm Pic}^1(C_1)$ is generically simply ramified along the diagonal $\Delta\subset C_1\times C_1$ and the locus $I$ of hyperelliptic conjugate pairs, and it admits a triple ramification at the points $\Delta\cap I$. One has to exclude also the cases $x=z$, that is, $2x\sim 2y$ and $x\not=y$, since we do not assume a base point at $x$. 
We conclude that in this case the number of desired pairs $(x,y)\in C_1\times C_1$ is $(72-3)\cdot 6-6\cdot 5=384$.

Note that, since every element of the test surface $C_1\times C_2$ is a curve of compact type, one can show that the intersection of $\overline{\mathcal{H}}^+_4$ with $C_1\times C_2$ is transverse at every point (see for instance \cite[Lemma 3.4]{MR985853}). It follows that the restriction of  $\overline{\mathcal{H}}^+_4$ to this test surface is
\[
\left[ \overline{\mathcal{H}}^+_4\right] \Big|_{V_1} = 2 \cdot 384\cdot 6 = 4608.
\]
Hence, we deduce the following relation
\[
18432 = 36 \cdot m + 4608 \cdot n -24 \cdot k.
\]

\subsubsection{}
\label{H4-2}
Let $C$ be a general curve of genus $2$. Let us consider the surface obtained by attaching two elliptic tails $E_1, E_2$ at two varying points $y_1, y_2$ in $C$. The base of this family is $C\times C=:V_2$. Let $\pi_i\colon C\times C \rightarrow C$ be the natural projection on the $i$-th component, for $i=1,2$.

\begin{figure}[htbp]
\centering
\psfrag{C}[b][b]{$C$}
\psfrag{E1}[b][b]{$E_1$}
\psfrag{E2}[b][b]{$E_2$}
 \includegraphics[scale=0.7]{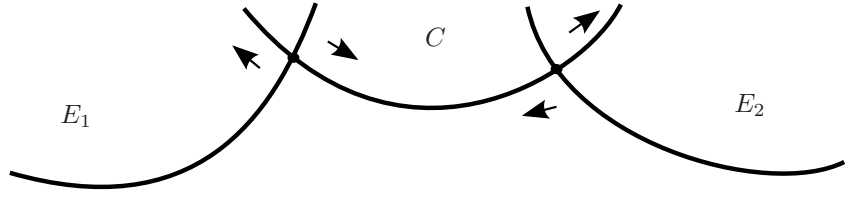}
  \caption{How the general fiber of the family in \S \ref{H4-2} moves.}
\end{figure}

The non-zero restrictions of the generating classes are
\begin{eqnarray*}
\delta^2_1 |_{V_2} &=& ({}-\pi_1^* K_C -\pi_2^* K_C -2\Delta_{C\times C} )^2 = 16,\\
\delta^2_2 |_{V_2} &=& \Delta^2_{C\times C}= -2,\\
\delta_{1|1} |_{V_2} &=& ({}-\pi_1^* K_C -\Delta_{C\times C})( {}-\pi_2^* K_C -\Delta_{C\times C} )= 6.
\end{eqnarray*}
We deduce that
\begin{align*}
\left[\overline{\Theta}_{\rm null}\right] \cdot \left[\overline{\mathcal{T}}\right] |_{V_2} &= 16896, &  \left[ \overline{\mathcal{H}yp}_4\right] |_{V_2} &= 30, & \left[\overline{\mathcal{W}}_2 \right] |_{V_2} &= 6, & \gamma_1 |_{V_2} &=0.
\end{align*}

If an element $E_1 \cup_{y_1} C \cup_{y_2} E_2$ of this family lies in $\overline{\mathcal{H}}^+_4$, then it admits a limit $\mathfrak{g}^1_{3}$
\[
(l_{E_1}, l_{C}, l_{E_2})=((\mathcal{L}_{E_1}, V_{E_1}),(\mathcal{L}_{C}, V_{C}),(\mathcal{L}_{E_2}, V_{E_2}))
\]
with $\mathcal{L}_{E_i}=\eta_{E_i}\otimes \mathcal{O}(3y_i)$, for $i=1,2$, and $\mathcal{L}_{C}=\eta_{C}\otimes \mathcal{O}(y_1+y_2)$, where $\eta_{E_1}, \eta_C, \eta_{E_2}$ are theta characteristics respectively on $E_1, C, E_2$, either all even, or two odd and one even. Moreover, one of the linear series $\mathfrak{g}^1_{3}$ has a section vanishing with order $3$ at some point $x$.
As in \S \ref{H4-1}, $x$ cannot specialize to a singular point.

Suppose that $x$ is in $E_1$ (the case $x$ in $E_2$ is similar, hence the answer will be multiplied by $2$). Since $a_1^{l_{E_1}}(x)=3$, we have $a_0^{l_{E_1}}(y_1)=0$, hence $a_1^{l_C}(y_1)=3$ and $a_0^{l_{C}}(y_2)=0$. It follows that $a_1^{l_{E_2}}(y_2)=3$, hence $h^0(\eta_{E_2})\geq 1$. This implies that $\eta_{E_2}$ is an odd theta characteristic $\eta_{E_2}^-$. There are two cases.

Let us consider the case when $\eta_{E_1}=\eta_{E_1}^+$ is an even theta characteristic and $\eta_{C}=\eta_{C}^-$ is an odd theta characteristic. Since $h^0(\eta_{E_1}^+)=0$, we have $a_1^{l_{E_1}}(y_1)\leq 2$, hence necessarily $a^{l_{E_1}}(y_1)=(0,2)$. This implies $a^{l_{C}}(y_1)=(1,3)$. It follows that $y_1$ is a Weierstrass point in $C$ and $l_C=y_1+|2y_1|$. Then $a^{l_{C}}(y_2)=(0,2)$, hence $y_2$ is also a Weierstrass point in $C$. It follows that  $l_{E_2}=y_2+|2y_2|$. 
Regarding the $E_1$ aspect, there exists $z\not= y_1$ in $E_1$ such that $\mathcal{O}(3x)=\mathcal{O}(2y_1+z)=\eta_{E_1}^+\otimes\mathcal{O}(3y_1)$. Hence, we have $\eta_{E_1}^+=\mathcal{O}(z-y_1)$ and $x$ satisfies $3x\sim 3z$, $x\not=z$. We conclude that this case gives the contribution $2\cdot (6\cdot 5)\cdot 3\cdot 8=1440$ to the intersection of this family with $\overline{\mathcal{H}}^+_4$. 

Next, we consider the case when $\eta_{E_1}=\eta_{E_1}^-$ is an odd theta characteristic and $\eta_{C}=\eta_{C}^+$ is an even theta characteristic. Again $a_0^{l_{E_1}}(y_1)=0$, hence $a^{l_{E_1}}(y_1)=(0,3)=a^{l_{C}}(y_1)$ (note that $l_{E_1}$ must have no base points). It follows that $a_0^{l_C}(y_2)=0$. Note that $a_1^{l_C}(y_2)\geq 2$.  Since $h^0(\mathcal{L}_C\otimes \mathcal{O}(-3y_1))\geq 1$, from \S \ref{Scorza} we deduce that $a^{l_C}(y_2)=(0,2)$, hence $l_{E_2}=y_2+|2y_2|$. Moreover, since $h^0(\eta_{C}^+\otimes \mathcal{O}(y_2-2y_1))\geq 1$, we have $\eta_{C}^+=\mathcal{O}(2y_1-y_2)$. From Proposition \ref{diff}, for each $\eta_C^+$ the number of such pairs is $8$. On $E_1$, since $h^0(\mathcal{O}(3y_1-3x))\geq 1$, we have that $y_1-x$ is a nontrivial $3$-torsion point in ${\rm Pic}^0(E_1)$. Hence the total contribution from this case is $2\cdot 8\cdot 8\cdot 10=1280$.

Let us suppose that $x$ is in $C$. We have necessarily $a_0^{l_C}(y_i)=0$ and $a_1^{l_{E_i}}(y_i)=3$. In particular $h^0(\mathcal{L}_{E_i}\otimes\mathcal{O}(-3y_1))\geq 1$, hence $\eta_{E_i}=\eta_{E_i}^-$ is an odd theta characteristic for $i=1,2$. It follows that $\eta_C=\eta_C^+$ is an even theta characteristic.  
From $h^0(\mathcal{L}_C\otimes \mathcal{O}(-2y_i))\geq 1$ for $i=1,2$, we deduce that the pair $(y_1,y_2)$ belongs to the Scorza curve
$ T_{\eta^+_C} $ (see \S \ref{Scorza}),
hence necessarily $y_1\not= y_2$.
Moreover, $x\in C$ satisfies $h^0(\mathcal{L}_C\otimes \mathcal{O}(-3x))\geq 1$, hence $\eta^+_C=\mathcal{O}(3x-y_1-y_2)$. It remains to count the number of triples $(x,y_1,y_2)\in C\times C\times C$ such that $(y_1,y_2)$ belongs to the Scorza curve $T_{\eta^+}$ with $\eta^+=\mathcal{O}(3x-y_1-y_2)$.
The class of $T_{\eta^+}$ in $H^2(C\times C)$ is in \S \ref{Scorza} after \cite{MR1213725}.
Consider $f\colon C \times C \times C \to \Pic^1(C)$ sending $(x, y_1, y_2)$ to $\mathcal{O}(3x - y_1 - y_2)$. 
Let $\pi_{i,j}\colon  C \times C \times C \to C\times C$ be the projection to the $i$-th and $j$-th factors, for $i,j\in\{1,2,3\}$. We want to compute 
\[
\deg (\pi_{2,3}^{*} \left[T_{\eta^{+}}\right] \cdot f^{*} [\eta^{+}]) =\deg (f_{*} (\pi_{2,3}^{*} F_1 + \pi_{2,3}^{*} F_2 + \pi_{2,3}^{*} \Delta_{C\times C})\cdot [\eta^{+}]).
\]
Note that $f$ restricted to $\pi_{2,3}^{*} F_1$ is the map $C\times C \to \Pic^1(C)$ sending 
$(x, y)$ to $\mathcal{O}(3x - y - q)$ where $q\in C$ is a fixed point. From Proposition \ref{diff2} we have that $\deg f_{*} (\pi_{2,3}^{*} F_1) \cdot [\eta^{+}] = 3^2 \cdot 1^2\cdot 2 = 18$. 
Similarly, we have $\deg f_{*} (\pi_{2,3}^{*} F_2) \cdot [\eta^{+}] = 18$. Finally, $f$ restricted to $\pi_{2,3}^{*} \Delta_{C\times C}$ is the map $C\times C\to \Pic^1(C)$ sending 
$(x, y)$ to $\mathcal{O}(3x - 2y)$, which from Proposition \ref{diff} has degree $3^2 \cdot 2^2 \cdot 2 = 72$. 
We conclude that 
\[ 
\deg (\pi_{2,3}^{*} \left[T_{\eta^{+}}\right] \cdot f^{*} [\eta^{+}] ) = 18 + 18 + 72 = 108.  
\]
Moreover, we have to exclude the cases $x=y_1$ or $x=y_2$. The map $f$ is simply ramified at $\pi_{1,2}^*\Delta_{C\times C}$ and $\pi_{1,3}^*\Delta_{C\times C}$, and again by Proposition \ref{diff} the degree of the restriction of $f$ to these two loci is $8$. We have to exclude also the cases when $x$ is a base point, that is, $\mathcal{O}(3x)=\mathcal{O}(2y_1+x)=\mathcal{O}(2y_2+x)$ and $\eta_{C}^+=\mathcal{O}(x-y_1+y_2)$. This happens when $x,y_1,y_2$ are different Weierstrass points.  We conclude that the contribution given by the case $x\in C$ is
\[
(108- 2\cdot2\cdot8)\cdot 10-6\cdot 5\cdot 4=640.
\]
Note that $a^{l_{E_i}}=(0,3)$ or $(1,3)$, and in each case the aspect $l_{E_i}$ is uniquely determined, for $i=1,2$.

It is easy to see that the fibers of this family over the diagonal $\Delta_{C\times C}$ are disjoint from $\overline{\mathcal{H}}^+_4$. 
We have thus shown that on this family
\[
 \left[\overline{\mathcal{H}}^+_4 \right] \Big|_{V_2} = 1440 + 1280 + 640 =3360.
\]
Hence, we have the following relation
\[
16896 = 30 \cdot m + 3360 \cdot n + 6 \cdot k.
\]

\subsubsection{} 
\label{H4-6}
 Attach at an elliptic curve $F$ a general curve $C$ of genus $2$ and an elliptic tail. Consider the surface obtained by varying $F$ in a pencil of degree $12$, and by varying one of the singular points on $F$.

\begin{figure}[htbp]
\centering
 \includegraphics[scale=0.7]{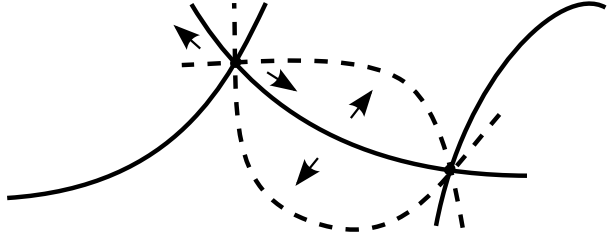}
  \caption{How the general fiber of the family in \S \ref{H4-6} moves.}
\end{figure}

The base of this family is the blow-up $V_3$ of $\mathbb{P}^2$ in the nine points of intersection of two general cubics, as in \S \ref{TS13Mbar31} (see also \cite[\S 3 ($\lambda$)]{MR1078265}). With the same notation from \S \ref{TS13Mbar31}, we have
\begin{align*}
\lambda |_{V_3} &= 3H-\Sigma, & \delta_0 |_{V_3} &= 36 H -12 \Sigma, &
\delta_1 |_{V_3} &= {}-3H+\Sigma, & \delta_2 |_{V_3} &= {}-3H + \Sigma - E_0.
\end{align*}
Hence, we deduce
\begin{align*}
\delta_1\delta_2 |_{V_3} &= 1, & \lambda\delta_2 |_{V_3} &= -1, & \delta_2^2 |_{V_3} &= 1.
\end{align*}
Moreover, we have
\begin{align*}
\delta_{01a} |_{V_3} &= 12, & \gamma_1 |_{V_3} &= 12.
\end{align*}
Indeed, there are $12$ fibers of this family contributing to the intersection with $\delta_{01a}$, namely when $F$ degenerates to a rational nodal curve and the moving node collides with the other non-disconnecting node. Similarly, there are $12$ fibers contributing to the intersection with $\gamma_1$, namely when $F$ degenerates to a rational nodal curve and the moving node collides with the disconnecting node. Each of these fibers contributes with multiplicity one.

The above restrictions imply
\begin{align*}
\left[\overline{\Theta}_{\rm null}\right] \cdot \left[\overline{\mathcal{T}}\right] |_{V_3} &= 2304, &   \left[\overline{\mathcal{H}yp}_4 \right] |_{V_3} &= 0, & \left[\overline{\mathcal{W}}_2\right] |_{V_3} &=-3.
\end{align*}

Since we assume that the singular point $q$ in $C$ is a general point in $C$, this surface has empty intersection with $\overline{\mathcal{H}}^+_4$. Indeed, suppose an element of this family is in $\overline{\mathcal{H}}^+_4$. Then, such an element admits a limit $\mathfrak{g}^1_{3}$ with one of the aspects vanishing to order $3$ at a certain point $x$. Note that the line bundle of the $C$-aspect is $\mathcal{L}_C=\eta_C\otimes \mathcal{O}(2q)$, for a certain theta characteristic $\eta_C$ of $C$. If $x$ is in $C$, then one has $h^0(\mathcal{L}_C\otimes\mathcal{O}(-3x))\geq 1$, hence $\eta_C=\mathcal{O}(3x-2q)$, a contradiction, since $q$ is general in $C$ (see Proposition \ref{diff}). If $x$ is not in $C$, then necessarily $h^0(\mathcal{L}_C\otimes\mathcal{O}(-3q))\geq 1$, hence $q$ is a Weierstrass point in $C$, a contradiction.

We deduce the following relation
\[
2304 = {}-3\cdot k + 12 \cdot l.
\]

\subsubsection{} 
\label{H4-8}
Let $(R, q_1, q_2, q_3, q_4, q_5)$ be a $5$-pointed rational curve. Attach at $q_1$ a general curve $C$ of genus $2$, and identify $q_2$ with $q_3$, and $q_4$ with $q_5$. Consider the family of curves obtained by varying the two moduli of the $5$-pointed rational curve.

\begin{figure}[htbp]
\centering
 \includegraphics[scale=0.7]{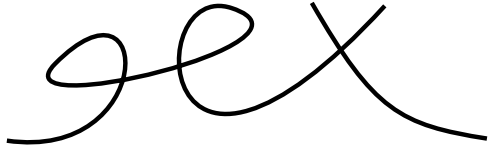}
  \caption{The general fiber of the family in \S \ref{H4-8}.}
\end{figure}

Since the point $q_1$ is general in $C$, this surface is disjoint from $\overline{\mathcal{H}}^+_4$. The argument is similar to the one in \S \ref{H4-6}. 

The base of this family is $\BM_{0,5}=:V_4$. We denote by $D_{i,j}$ the divisor in $\BM_{0,5}$ of curves with two rational components meeting transversally in a point, with the marked points $i$ and $j$ in a component, and the other three points in the other component, for $i,j\in\{1,2,3,4,5\}$. We denote by $\psi_i$ the cotangent line class in $\BM_{0,5}$ corresponding to the point marked by $i$, for $i\in \{1,2,3,4,5\}$. Let $j,k$ be two markings different from $i$, and let $l,m$ be the other two markings. The following relation is well-known
\[
\psi_i = D_{j,k} + D_{i,l} + D_{i,m}.
\]

We have
\begin{eqnarray*}
\lambda |_{V_4} &=& 0,\\
\delta_0 |_{V_4} &\equiv& {}-\psi_2 -\psi_3 -\psi_4 -\psi_5 + D_{1,2} + D_{1,3} + D_{1,4} + D_{1,5} + D_{2,4} + D_{2,5} + D_{3,4} + D_{3,5}\\
	&\equiv& {}-2D_{4,5} -D_{2,3} -D_{3,4} -D_{1,2} + D_{1,4},\\
\delta_1 |_{V_4} &\equiv& D_{4,5} + D_{2,3} \\
\delta_2 |_{V_4} &=& -\psi_1.
\end{eqnarray*}
Using $D_{i,j}^2=-1$, $D_{i,j}\cdot D_{k,l}=1$ if $\{i,j\}\cap \{k,l\}=\emptyset$, and $D_{i,j}\cdot D_{k,l}=0$ otherwise, we have
\begin{align*}
\delta_2^2 |_{V_4} &= 1, & \delta_1\delta_2 |_{V_4} &= -2,& \delta_{1|1} |_{V_4} &= 1.
\end{align*}
We also have
\begin{align*}
\gamma_1 |_{V_4} \equiv (D_{1,2}+ D_{1,3}) ({}-\psi_2 -\psi_3) + (D_{1,4}+ D_{1,5}) ({}-\psi_4 -\psi_5) + D_{2,4}\cdot D_{3,5} + D_{2,5}\cdot D_{3,4} = -2.
\end{align*}
Indeed, along the divisors $D_{1,2}, D_{1,3}, D_{1,4}, D_{1,5}$ when the disconnecting node collides with a non-dis\-connecting node, the fibers of this family all contribute to the intersection with $\gamma_1$. By the excess intersection formula, these contributions equal the restriction of the normal bundle at the two points corresponding to the non-disconnecting node.
Moreover, the fibers of this family in the intersections $D_{2,4}\cdot D_{3,5}$ and $D_{2,5}\cdot D_{3,4}$ give additional contributions.

We deduce
\begin{align*}
\left[\overline{\Theta}_{\rm null}\right] \cdot \left[\overline{\mathcal{T}}\right] |_{V_4} &= -528, &   \left[\overline{\mathcal{H}yp}_4\right] |_{V_4} &= 0, & \left[\overline{\mathcal{W}}_2 \right] |_{V_4} &=-1,
\end{align*}
and the following relation follows
\[
-528 = {}- k -2 \cdot l.
\]

\subsection{Proof of Theorem \ref{h4+}} 
In \S \ref{intThT}, we have seen that the locus $\overline{\mathcal{H}}^+_4$ is a component of the intersection of the divisors $\overline{\Theta}_{\rm null}$ and $\overline{\mathcal{T}}$, and we have analyzed the other components of the intersection. To compute the class of $\overline{\mathcal{H}}^+_4$, it remains to compute the coefficients $m,n,k,l$ in (\ref{mnkl}). Restricting to the test surfaces in \S \ref{TS4}, we have found the four linear relations 
\begin{eqnarray*}
18432 &=& 36 \cdot m + 4608 \cdot n -24 \cdot k,\\
16896 &=& 30 \cdot m + 3360 \cdot n + 6 \cdot k,\\
2304 &=& {}-3\cdot k + 12 \cdot l,\\
-528 &=& {}- k -2 \cdot l,
\end{eqnarray*}
whence we deduce that
$m = 320$,  $n = 2$,  $k = 96$, and $l = 216$.
The class of $\overline{\mathcal{H}}^+_4$ follows.
\hfill$\square$

\section{Complete intersections}
\label{ci}

In this section, we prove that the loci $\overline{\mathcal{H}yp}_{3,1}$, $\overline{\mathcal{F}}_{3,1}$, and $\overline{\mathcal{H}}^+_4$ are not complete intersections in their respective spaces.

\begin{proof}[Proof of Corollary \ref{coro}]
Modulo the relation $(5 \lambda + \psi - \frac{1}{2}\delta_0 -\delta_{2,1} ) \delta_{1,1} =0$ in $R^2(\BM_{3,1})$ (see Proposition \ref{boundaryinA2Mbar31}), the product of two divisor classes in $\BM_{3,1}$ can be written in terms of the basis of $R^2(\BM_{3,1})$ in Proposition \ref{basisR2Mbar31}.
The resulting coefficient of $\kappa_2$ is zero. From Theorem \ref{res31}, it follows that $\overline{\mathcal{F}}_{3,1}$ is not a complete intersection in $\BM_{3,1}$.

Similarly, modulo the relation $(10\lambda - \delta_0 - 2\delta_1)\delta_2=0$ in $R^2(\BM_4)$, the product of two divisor classes in $\BM_4$ can be written in terms of the basis in (\ref{basisA2M4}). The resulting coefficients of the classes $\delta_{00}$, $\gamma_1$, $\delta_{01a}$, and 
$\delta_{1|1}$ are zero, hence from Theorem \ref{h4+}, $\overline{\mathcal{H}}^+_4$ is not a complete intersection in $\BM_4$.

Finally, suppose that the class of $\overline{\mathcal{H}yp}_{3,1}$ is a product of two {\it effective} divisor classes. 
Imposing the product of two arbitrary divisor classes in $\BM_{3,1}$ to be equal to the class of $\overline{\mathcal{H}yp}_{3,1}$ in Theorem \ref{res31} (modulo $(5 \lambda + \psi - \frac{1}{2}\delta_0 -\delta_{2,1} ) \delta_{1,1} =0$) yields $14$ relations in the $10$ coefficients of the two arbitrary divisor classes.
This forces the two divisor classes to be a multiple of $p^*[\overline{\mathcal{H}yp}_{3}]$ and a multiple of $D:=2\psi-5\lambda+\frac{1}{2}\delta_0 + \delta_2$. The class of the Weierstrass divisor $\overline{\mathcal{W}}_{3,1}\equiv 6\psi-\lambda-3\delta_1-\delta_2$ is inside the cone generated by $D$ and the effective classes $\psi$ and $p^*[\overline{\mathcal{H}yp}_{3}]$. This contradicts the extremality of the class of $\overline{\mathcal{W}}_{3,1}$ (see \cite{MR3034451} or \cite{MR3071469}). It follows that $D$ is not effective, hence $\overline{\mathcal{H}yp}_{3,1}$ is not a complete intersection in $\BM_{3,1}$.
\end{proof}

\begin{rem}
Using the class of $\overline{\mathcal{H}yp}_{4}$ computed in \cite[Proposition 5]{MR2120989} (see (\ref{hyp4})), the above argument for $\overline{\mathcal{H}}^+_4$ shows that $\overline{\mathcal{H}yp}_{4}$ is also not a complete intersection in $\BM_4$.
\end{rem}

\section{The determinantal description of \texorpdfstring{${\mathcal{H}}_4$}{H4} and \texorpdfstring{${\mathcal{H}}^-_4$}{H-4}}

As a partial check on Theorem \ref{h4+}, in this section we compute the coefficient of $\lambda^2$ in the class of $\overline{\mathcal{H}}^+_4$ using a determinant description of ${\mathcal{H}}_4$ in $\MM_4$. At the same time, we  also compute the coefficient of $\lambda^2$ in the class of $\overline{\mathcal{H}}^-_4$.

\subsection{The locus \texorpdfstring{$\HH^-_4$}{H-4}}
\label{det-}

Let $\mathcal{SH}^-_4$ be the locus in $\mathcal{S}^-_4$  of odd spin genus-$4$ curves $[C,\eta_C^-]$ such that the natural map
\[
 \varphi \colon H^0(C,\eta^-_C) \rightarrow H^0(C, \eta^-_C|_{3x})
\]
has rank zero for some point $x$ in $C$. Note that the locus ${\HH}^-_4$ in $\mathcal{M}_4$ is the push-forward of $\mathcal{SH}^-_4$ via the natural map $\pi\colon \mathcal{S}^-_4 \rightarrow  \MM_4$.

Let $p\colon \mathcal{C}\rightarrow \mathcal{S}^-_4$ be the universal curve and $\eta^-\in {\rm Pic}(\mathcal{C})$ be the universal spin bundle of relative degree $g-1$. Note that $p_* \eta^-$ is a line bundle outside a locus of codimension at least $3$ in $\mathcal{S}^-_4$. We can ignore such locus since we will only deal with Chern classes $c_i(p_* \eta^-)$ with $i<3$. In particular $c_2(p_* \eta^-)=0$. The map $\varphi$ globalizes to a map of vector bundles
\[
 \widetilde{\varphi} \colon p^* p_* \eta^- \rightarrow J_2(\eta^-)
\]
respectively of rank $1$ and $3$ over $\mathcal{C}$, where $J_2(\eta^-)$ is the second jet bundle of $\eta^-$. We are interested in the locus of curves $C$ where $\widetilde{\varphi}$ has rank zero. By Porteous formula, we have
\[
 [\mathcal{SH}^-_4]= p_* c_3 \left( J_2(\eta^-) - p^* p_* \eta^- \right) \in A^2(\mathcal{S}^-_4).
\]

Let us compute the Chern classes of $p^* p_* \eta^-$ and $J_2(\eta^-)$.
Since $(\eta^-)^{\otimes 2} \simeq \omega_p$, we deduce $ch(\eta^-)=e^{\frac{1}{2}\psi}$, where $\psi=c_1(\omega_p)$. From Grothendieck-Riemann-Roch, we have
\begin{eqnarray*}
 ch(p_* \eta^-)-ch(R^1 p_* \eta^-) &=& p_* \left( td(\omega_p^\vee) \cdot ch(\eta^-) \right)\\
	&=& p_* \left( \frac{\psi}{e^\psi -1} \cdot e^{\frac{1}{2}\psi} \right)\\
	&=& p_* \left( (1-\frac{1}{2}\psi +\frac{1}{12}\psi^2 -\frac{1}{720}\psi^4+\cdots) \cdot e^{\frac{1}{2}\psi} \right)\\
	&=& p_* \left( 1 -\frac{1}{24}\psi^2 +\frac{7}{5760}\psi^4 +\cdots \right)\\
	&=& {}-\frac{1}{24}\kappa_1 +\frac{7}{5760}\kappa_3 +\cdots. 
\end{eqnarray*}
In particular, we have
\[
 2c_1(p_* \eta^-) = {}-\frac{1}{24}\kappa_1 = {}-\frac{1}{2}\lambda.
\]
From the standard exact sequence
\[
 0\rightarrow \eta^- \otimes Sym^n \omega_p \rightarrow J_n (\eta^-) \rightarrow J_{n-1} (\eta^-) \rightarrow 0
\]
we compute
\begin{align*}
 ch(J_2(\eta^-)) = ch(\eta^- \otimes (1 \oplus \omega_p \oplus Sym^2 \omega_p))
	= e^{\frac{1}{2}\psi} \cdot (1 +e^\psi +e^{2\psi})
	= 3 +\frac{9}{2}\psi +\frac{35}{8}\psi^2 +\frac{51}{16}\psi^3 +\!\cdots
\end{align*}
whence we deduce
\[
 c (J_2(\eta^-)) = 1 + \frac{9}{2}\psi +\frac{23}{4}\psi^2 + \frac{15}{8}\psi^3.
\]
 We obtain
\begin{eqnarray*}
 [\mathcal{SH}^-_4]  &=& p_* \left(c_3( J_2(\eta^-)) - c_1(p_* \eta^-)c_2( J_2(\eta^-)) +c_1^2(p_* \eta^-)c_1( J_2(\eta^-)) \right)\\
	&=& \frac{15}{8} \kappa_2 +\frac{23}{16} \kappa_1\lambda + \frac{27}{16}\lambda^2 = \frac{177}{4}\lambda^2.
\end{eqnarray*}
In the last equality we have used Mumford's relation $\kappa_1 = 12\lambda$ and Faber's relation $\kappa_2=\frac{27}{2}\lambda^2$ on $\mathcal{M}_4$.
Since the degree of $\pi$ is $2^{g-1}(2^g-1)$, we deduce
\[
  \HH^-_4 \equiv 5310 \lambda^2 \in A^2(\mathcal{M}_4).
\]

\subsection{The locus \texorpdfstring{$\HH_4$}{H4}}
Let us consider the locus
\[
 \HH_4 :=\{[C]\in \mathcal{M}_4 :  h^0(K_C(-6x))\geq 1 \,\, \mbox{for some}\,\, x\in C  \}.
\]
The locus $\HH_4$ consists of curves $C$ of genus $4$ such that the natural map
\[
 \varphi \colon H^0(K_C) \rightarrow H^0(K_C|_{6x})
\]
has rank at most $3$, for some point $x$ in $C$. Let $p\colon \mathcal{C}\rightarrow \mathcal{M}_4$ be the universal curve. The map $\varphi$ globalizes to a map of vector bundles
\[
 \widetilde{\varphi}\colon p^* E \rightarrow J_5(\omega_p)
\]
where $E:=p_*(\omega_p)$ is the Hodge bundle of rank $4$ and $J_5(\omega_p)$ is the $5^{\rm th}$ jet bundle of $\omega_p$. Using Porteous formula, we have
\[
 [\HH_4] = p_* c_3\left(J_5(\omega_p) - p^*E \right) \in A^2(\mathcal{M}_4).
\]
Note that
\begin{eqnarray*}
 ch(J_5(\omega_p)) &=& ch(\omega_p\otimes (1\oplus \omega_p \oplus Sym^2 \omega_p \oplus \cdots \oplus Sym^5 \omega_p))
	=\sum_{i=1}^6 e^{i\psi}\\
	&=& 6 +21\psi +\frac{91}{2}\psi^2 +\frac{441}{6}\psi^3 +\cdots.
\end{eqnarray*}
Hence 
\[
 c(J_5(\omega_p)) = 1 +21\psi +175\psi^2 +735\psi^3 +\cdots
\]
and we obtain
\begin{eqnarray*}
 [\HH_4] &=& p_* \left( 735\psi^3 -175\psi^2\lambda_1 +21\psi(\lambda_1^2-\lambda_2) \right)
	= 735 \kappa_2 -175\kappa_1\lambda_1 +21\cdot (2g-2)(\lambda_1^2-\lambda_2)\\
	&=& \frac{15771}{2}\lambda_1^2.
\end{eqnarray*}
In the last equality we have used the relations $\kappa_1=12\lambda_1$, $\lambda_2=\frac{\lambda_1^2}{2}$, and $\kappa_2=\frac{27}{2}\lambda_1^2$ on $\mathcal{M}_4$.

We have the following equality
\[
\left[ \HH_4 \right] = 10\cdot \left[\mathcal{H}yp_4 \right] + \left[\HH^+_4 \right]+ \left[\HH^-_4 \right].
\]
The multiplicity $10$ is due to the fact that each of the $10$ Weierstrass points gives a contribution.
Hence we deduce
\begin{align}
\label{checkH4+}
  \HH^+_4 \equiv \frac{15771}{2}\lambda_1^2 -10\cdot \frac{51}{4}\lambda^2 - 5310 \lambda^2 = 2448\lambda_1^2 \in A^2(\mathcal{M}_4)
\end{align}
 and this checks with Theorem \ref{h4+}.

\bibliographystyle{alpha}
\bibliography{biblio.bib}

\end{document}